\documentclass[12pt]{article}
\usepackage{amsmath,amsxtra,amssymb,color,latexsym,epsfig,amscd,amsthm,subfigure,fancybox,epsfig}
\usepackage[mathscr]{eucal}
\usepackage{graphicx}
\usepackage{enumerate}
\usepackage{enumitem}
\usepackage{epsfig} 
\usepackage{epstopdf}
\usepackage{cases}
\newtheorem{thm}{Theorem}[section]
\newtheorem {asp}{Assumption}[section]
\newtheorem{lm}{Lemma}[section]
\newtheorem{rmk}{Remark}[section]

\newtheorem{deff}{Definition}[section]

\theoremstyle{definition}

\theoremstyle{remark}

\newtheorem{exam}{Example}[section]
\numberwithin{equation}{section}

\DeclareMathOperator{\trace}{tr}
\DeclareMathOperator{\sgn}{sgn}
\newcommand{\eps}{\varepsilon}

\newcommand{\A}{\mathcal{A}}
\newcommand{\B}{\mathcal{B}}
\newcommand{\fB}{\mathfrak{B}}
\newcommand{\p}{\mathfrak{p}}
\newcommand{\m}{\mathfrak{m}}
\newcommand{\C}{\mathcal{C}}
\newcommand{\D}{\mathcal{D}}
\newcommand{\DD}{\mathbb{D}}

\newcommand{\F}{\mathcal{F}}

\newcommand{\E}{\mathbb{E}}
\newcommand{\fE}{\mathfrak{E}}
\newcommand{\CE}{\mathcal{E}}
\newcommand{\LL}{\mathcal{L}}

\newcommand{\N}{{\mathbb{Z}}_+}

\newcommand{\HB}{\mathbb{H}}

\newcommand{\PP}{\mathbb{P}}

\newcommand{\R}{\mathbb{R}}

\newcommand{\BF}{\mathbb{F}}

\numberwithin{equation}{section}
\newcommand{\1}{\boldsymbol{1}}

\newcommand{\bed}{\begin{displaymath}}
\newcommand{\eed}{\end{displaymath}}
\newcommand{\bea}{\bed\begin{array}{rl}}
\newcommand{\eea}{\end{array}\eed}
\newcommand{\ad}{&\!\!\!\disp}

\newcommand{\barray}{\begin{array}{ll}}
\newcommand{\earray}{\end{array}}

\def\disp{\displaystyle}

\def\bar{\overline}
\def\hat{\widehat}
\def\a.s{\text{\;a.s.\;}}

\def\bnu{\boldsymbol{\nu}}

\begin{document}
\title{Recurrence and Ergodicity of Switching Diffusions with Past-Dependent Switching Having A Countable State Space\thanks{This
research was supported in part by the National Science Foundation under grant DMS-1207667.}}
\author{Dang H. Nguyen\thanks{Department of Mathematics, Wayne State University, Detroit, MI
48202,
dangnh.maths@gmail.com.} \and
George Yin\thanks{Department of Mathematics, Wayne State University, Detroit, MI
48202,
gyin@math.wayne.edu.}}
\maketitle

\begin{abstract} This work focuses on recurrence and ergodicity of
switching diffusions consisting of continuous and discrete components, in which the discrete
component takes values in a countably infinite set and the rates of switching
at current time
depend on the value of the continuous component over an interval including certain past history.
Sufficient conditions for recurrence and ergodicity are given.
Moreover, the relationship between systems of partial differential equations and recurrence
when the switching is past-independent is established
under suitable conditions.



\end{abstract}

\newpage

\section{Introduction}\label{sec:int}
Emerging and existing applications in  wireless communications, queueing networks, biological models,
ecological systems, financial engineering, and social networks demand the
mathematical modeling, analysis, and computation
of hybrid systems in which continuous dynamics and discrete events coexist.
Switching diffusions are one of such hybrid models.
A switching diffusion is a two-component process $(X(t),\alpha(t))$, a continuous component and a discrete component taking values in a set
consisting of isolated points.
When the discrete component takes a value $\alpha$ (i.e., $\alpha(t)=\alpha$),
the continuous component $X(t)$ evolves according to the diffusion process whose drift and diffusion coefficients depend on $\alpha$.
Such processes
 have received growing attention recently because of their ability to a wide range of applications; see \cite{KZY,XM,SX,ZY}
and the references therein.

In the comprehensive treatment of hybrid switching diffusions in \cite{MY}, it was assumed that
  $\alpha(t)$ is a continuous-time and homogeneous Markov chain  independent of the Brownian motion and that
 the generator of the Markov chain is a constant matrix.
 For broader impact on applications, considering the two components jointly,
  the work \cite{YZ} extended the study to the
 Markov process  $(X(t),\alpha(t))$
by allowing the generator of $\alpha(t)$ to depend on the current state $X(t)$.
Until very recently, most of
the works treat
 $\alpha(t)$ as a process taking values in a finite set.
Even when $\alpha(t)$ is allowed to take values in
 a countable state space, almost all works required
 the systems being memoryless. That is, the switching depends on the continuous state, with the dependence on
  the current continuous state only, no delays are involved; see, for example, \cite{MY,SJ,SX,YZ} and references therein.
To be able to treat more complex models and to broaden the applicability,
we have undertaken the task of
 investigating the dynamics of $(X(t),\alpha(t))$ in which $\alpha(t)$ has a countable state space
and its switching intensities depend on the history of the continuous component  $X(t)$.
As a first attempt, this type of switching diffusion was considered in \cite{DY}, which was
 motivated by queueing and control systems applications.
In particular,  the evolution of two interacting species was considered in the aforementioned reference. One of the species is micro
described by a logistic differential equation perturbed by a white noise, and
the other is macro.
  Let $X(t)$ be the density of the micro species and $\alpha(t)$ the population of the macro species.
 The reproduction process of $\alpha(t)$ is
 non-instantaneous, resulting in past-dependent switching.
In \cite{DY}, we gave precise formulation of the process  $(X(t),\alpha(t))$ and established
the existence and uniqueness of solutions together with such
 properties as Markov-Feller property
 and Feller property of  function-valued stochastic processes associated with our processes
under suitable conditions.

Many real-world systems are in operation for a long period
of time. Similar to their diffusion counter part, longtime behavior of switching diffusion systems is very important.
When the switching is independent of past, a number of results
have been obtained in \cite{BSY, SJ, SX, SX2, ZY}. Nevertheless, when past-dependent switching is considered, not much of
 the desired asymptotic properties are known to the best of our knowledge.
Motivated by the practical needs,
this paper studies recurrence and ergodicity of
switching diffusion processes with past-dependent switching having a countable state space.
Such systems are more difficult to handle. To begin, the pair $(X(t),\alpha(t))$ is no longer a Markov process because of past-dependence of the
switching process.
Most of the arguments based on Markov property cannot be used for the aforementioned pair, so different approaches have to be taken.
The next idea is to treat the process $(X_t,\alpha(t))$, where $X_t$ is the so-called segment process associated with $X(t)$. For such processes, although we have Markov properties,
 the systems that we face become infinite dimensional.
As will be seen in the subsequent sections, the problems require much more attention and careful consideration.

The rest of the paper is organized as follows. The formulation of switching diffusions with past-dependent switching and
countably many possible switching locations is given in Section \ref{sec:for}.
Also some relevant results, including
the functional It\^o formula, are recalled briefly. These results
 play a crucial role.
Section \ref{sec:3} provides certain sufficient conditions for recurrence, positive recurrence, and ergodicity of
the related Markov process associated with our switching diffusions.
In Section \ref{sec:4}, we characterize the recurrence of switching diffusion that are past independent.
Furthermore,
the relationship between recurrence and the associated systems of partial differential equations
is established.
Finally, we provide the proofs of some technical results in an appendix.

\section{Formulation}\label{sec:for}
Denote by $\C([a,b],\R^n)$ the set of  $\R^n$-valued continuous functions defined on $[a, b]$. Let $r$ be a fixed positive number.
In what follows,
we mainly work with $\C([-r,0],\R^n)$, and simply denote $\C:=\C([-r,0],\R^n)$.
For $\phi\in\C$, we use the norm $\|\phi\|=\sup\{|\phi(t)|: t\in[-r,0]\}$.
For $t\geq0$, we denote  by $y_t$
the
so-called segment function (or memory segment function) $y_t =\{ y(t+s): -r \le s \le 0\}$.
For $x\in\R^n$, denote by $|x|$ the Euclidean norm of $x$.
Let $\fB(\C)$ and $\fB(\C\times\N)$ be the $\sigma$-algebras
on $\C$ and $\C\times\N$ respectively.
Let $(\Omega,\F,\{\F_t\}_{t\geq0},\PP)$ be a complete filtered probability space with the filtration $\{\F_t\}_{t\geq 0}$ satisfying the usual condition,  i.e., it is increasing and right continuous while $\F_0$ contains all $\PP$-null sets.
Let $W(t)$ be an $\F_t$-adapted and $\R^d$-valued Brownian motion.
Suppose $b(\cdot,\cdot): \R^n\times\N\to\R^n$ and $\sigma(\cdot,\cdot): \R^n\times\N\to\R^{n\times d}$.
Consider the two-component process $(X(t),\alpha(t))$ where
 $\alpha(t)$ is a pure jump process taking value in  $\N= {\mathbb N} \setminus \{0\}=\{1,2,\dots\}$, the set of positive integers, and $X(t)$ satisfies
\begin{equation}\label{eq:sde} dX(t)=b(X(t), \alpha(t))dt+\sigma(X(t),\alpha(t))dW(t).\end{equation}
We assume that the switching intensity of $\alpha(t)$ depends on the trajectory of $X(t)$ in the interval $[t-r,t]$,
that is, there are functions $q_{ij}(\cdot):\C\to\R$ for $i,j\in\N$ satisfying
$q_{ij}(\phi)\geq0\,\forall i\ne j$ and $\sum_{j=1,j\ne i}^\infty q_{ij}(\phi)=q_{i}(\phi)$ for all $\phi\in\C$,which means that $\alpha(t)$ is conservative. If $q_i(\phi)$ is uniformly bounded in $(\phi,i)\in\C\times\N$, and
$q_{i}(\cdot)$ and $q_{ij}(\cdot)$ are continuous, then we have the
following interpretation
\begin{equation}\label{eq:tran}\begin{array}{ll}
&\disp \PP\{\alpha(t+\Delta)=j|\alpha(t)=i, X_s,\alpha(s), s\leq t\}=q_{ij}(X_t)\Delta+o(\Delta) \text{ if } i\ne j \
\hbox{ and }\\
&\disp \PP\{\alpha(t+\Delta)=i|\alpha(t)=i, X_s,\alpha(s), s\leq t\}=1-q_{i}(X_t)\Delta+o(\Delta).\end{array}\end{equation}
[Note that in \eqref{eq:tran}, in contrast to \cite{YZ},
 in lieu of $X(t)$, $X_t$ is used.]
For more general $q_i(\cdot)$ and $q_{ij}(\cdot)$,
the process $\alpha(t)$ can be defined rigorously as the solution to a stochastic differential equation with respect to a Poisson random measure.
For each function $\phi: [-r,0]\to\R^n$, and $i\in\N$,  let $\Delta_{ij}(\phi), j\ne i$  be the consecutive left-closed, right-open intervals of the real line, each having length $q_{ij}(\phi)$.
That is
\bea \ad \Delta_{i1}(\phi)=[0,q_{i1}(\phi)),\\
\ad \Delta_{ij}(\phi)=\Big[\sum_{k=1,k\ne i}^{j-1}q_{ik}(\phi),\sum_{k=1,k\ne i}^{j}q_{ik}(\phi)\Big), \ j>1 \ \hbox{ and }\ j\ne i.\eea
Define $h:\C\times\N\times\R\mapsto\R$ by
$h(\phi, i, z)=\sum_{j=1, j\ne i}^\infty(j-i)\1_{\{z\in\Delta_{ij}(\phi)\}}.$
The process $\alpha(t)$ can be defined as the solution to
$$d\alpha(t)=\int_{\R}h(X_t,\alpha(t-), z)\p(dt, dz),$$	
where $a(t-)=\lim\limits_{s\to t^-}\alpha(s)$ and $\p(dt, dz)$ is a Poisson random measure with intensity $dt\times\m(dz)$ and $\m$ is the Lebesgue measure on $\R$ such that
$\p(dt, dz)$ is
independent of the Brownian motion $W(\cdot)$.
The pair $(X(t),\alpha(t))$ is therefore a solution to
\begin{equation}\label{e2.3}
\begin{cases}
dX(t)=b(X(t), \alpha(t))dt+\sigma(X(t),\alpha(t))dW(t) \\
d\alpha(t)=\disp\int_{\R}h(X_t,\alpha(t-), z)\p(dt, dz).\end{cases}
\end{equation}
With initial data $(\xi,i_0)$ being
a $\C\times\N$-valued and $\F_0$-measurable random variable,
a strong solution to  \eqref{e2.3} on $[0,T]$,
is an $\F_t$-adapted process $(X(t),\alpha(t))$ such that
\begin{itemize}
  \item $X(t)$ is continuous and $\alpha(t)$ is cadlag (right continuous with left limits) with probability 1 (w.p.1).
  \item $X(t)=\xi(t)$ for $t\in[-r,0]$ and  $\alpha(0)=i_0$
  \item $(X(t),\alpha(t))$ satisfies \eqref{e2.3} for all $t\in[0,T]$ w.p.1.
\end{itemize}
To ensure the existence and uniqueness of solutions,
we assume  that one of the following two assumptions holds throught this paper (see \cite{DY}).

\begin{asp}\label{asp2.3}{\rm
The following conditions hold.
\begin{itemize}
  \item[{(i)}] For each $H>0$, $i\in\N$, there is a positive constant $L_{H,i}$ such that
$$|b(x,i)-b(y,i)|+|\sigma(x,i)-\sigma(y,i)|\leq L_{H,i}|x-y|, \ \forall |x|, |y|\leq H, i\in\N.$$
\item[{(ii)}] For each $i\in\N$, there exists a  twice continuously differentiable function $V_i(x)$ and a constant $C_i>0$ such that
$$\lim\limits_{R\to\infty}\Big(\inf\{V_i(x):|x|\geq R\}\Big)=\infty\quad\text{ and }\quad\LL_iV_i(x)\leq C_i(1+V_i(x))\,\forall\,x\in\R^n,$$
where $\LL_i$ is the generator of the diffusion at the fixed state $i$, that is,
\begin{equation}\label{e.L_i}
\LL_i V(x)=\nabla V(x)b(x,i)+\dfrac12\trace\Big(\nabla^2 V(x)A(x,i)\Big)
\end{equation}
for $V(\cdot)\in C^2(\R^n)$, where $A(x,i)=\sigma(x,i) \sigma^{\sf T}(x,i)$. 
\item[{(iii)}] $q_{ij}(\phi)$ and $q_i(\phi)$ are continuous in $\phi\in\C$ for each $i,j\in\N, i\ne j$. Moreover,
      $$M:=\sup_{\phi\in\C, i\in\N}\{q_{i}(\phi)\}<\infty.$$
\end{itemize} }
\end{asp}

\begin{asp}\label{asp2.4} {\rm
The following conditions hold.
\begin{itemize}
  \item[{(i)}] For each $H>0$, $i\in\N$, there is a positive constant $L_{H,i}$ such that
$$|b(x,i)-b(y,i)|+|\sigma(x,i)-\sigma(y,i)|\leq L_{H,i}|x-y|, \, \ \forall |x|, |y|\leq H, \ i\in\N.$$
\item[{(ii)}] There exists a  twice continuously differentiable function $V(x)$ and a constant $C>0$ independent of $i\in\N$ such that
$$\lim\limits_{R\to\infty}\Big(\inf\{V(x):|x|\geq R\}\Big)=\infty\quad\text{ and }\quad\LL_iV(x)\leq C(1+V(x)), \,\forall\,x\in\R^n, \ i\in\N.$$
\item[{(iii)}]  $q_{ij}(\phi)$ and $q_i(\phi)$ are continuous in $\phi\in\C$ for each $i,j\in\N, i\ne j$. Moreover, for any $H>0$,
      $$M_H:=\sup_{\phi\in\C, \|\phi\|\leq H, i\in\N}\{q_{i}(\phi)\}<\infty.$$
\end{itemize}}
\end{asp}

It is proved in \cite{DY} that under either of the above two assumptions, the process
$(X_t,\alpha(t))$
satisfying \eqref{e2.3} is a Markov-Feller process.
Now we state the functional It\^o formula for our process (see \cite{CF} for more details).
Let $\DD$ be the space of cadlag functions $f:[-r,0]\mapsto\R^n$.
For $\phi\in\DD$, we define horizontal and vertical perturbations for $h\ge 0$ and $y\in \R^n$ as
$$
\phi_h(s)=
\begin{cases}\phi(s+h)\, \text{ if }\, s\in[-r,-h],\\
 \phi(0) \, \text{ if }\,s\in[-h,-0],
\end{cases}
$$
and
$$
\phi^y(s)=
\begin{cases}\phi(s)\, \text{ if }\, s\in[-r,0),\\
 \phi(0)+y,
\end{cases}
$$
respectively.
Let $V:\DD\times\N\mapsto\R$.
The horizontal derivative at $(\phi,i)$ and vertical partial derivative of $V$ are defined as
\begin{equation}\label{e.dt}
V_t(\phi,i)=\lim\limits_{h\to0} \dfrac{V(\phi_h,i)-V(\phi)}h
\end{equation}
and
\begin{equation}\label{e.dx}
\partial_i V(\phi,i)=\lim\limits_{h\to0} \dfrac{V(\phi^{he_i},i)-V(\phi)}h
\end{equation}
if these limits exist.
In \eqref{e.dx}, $e_i$ is the standard unit vector in $\R^n$ whose $i$-th component is $1$ and other components are $0.$
Let $\BF$ be the family of function $V(\cdot,\cdot):\DD\times\N\mapsto\R$ satisfying that
\begin{itemize}
  \item $V$ is continuous, that is, for any $\eps>0$, $(\phi,i)\in\DD\times\N$, there is a $\delta>0$ such that
  $|V(\phi,i)-V(\phi',i)|<\eps$ as long as $\|\phi-\phi'\|<\delta$.
\item The derivatives
$V_t$, $V_x=(\partial_k V)$, and $V_{xx}=(\partial_{kl} V)$ exist and are continuous.
\item $V$, $V_t$, $V_x=(\partial_k V)$ and $V_{xx}=(\partial_{kl} V)$ are bounded in each $B_R:=\{(\phi,i): \|\phi\|\leq R, i\leq R\}$, $R>0$.
\end{itemize}

\begin{rmk}\label{rmk2.1} {\rm
Recently, a functional It\^o formula was developed in \cite{Dup09}, which encouraged subsequent development (for example, \cite{CF,PH}). We briefly recall the main idea in what follows.
Consider functions of the form
 $$V(\phi,i)=f_1(\phi(0),i)+\int_{-r}^0g(t,i)f_2(\phi(t),i)dt,$$
where $f_2(\cdot,\cdot):\R^n\times\N\mapsto\R$ is a continuous function and $f_1(\cdot,\cdot):\R^n\times\N\mapsto\R$ is a function that is twice continuously differentiable in the first variable
and $g(\cdot,\cdot):\R_+\times\N\mapsto\R$ be a continuously differentiable function in the first variable.
Then at $(\phi,i)\in\C\times\N$ we have (see \cite{PH} for the detailed computations)
$$V_t(\phi,i)=g(0,i)f_2(\phi(0),i)-g(-r,i)f_2(\phi(-r),i)-\int_{-r}^0 f_2(\phi(t),i)dg(t,i),$$
$$\partial_k V(\phi,i)=\dfrac{\partial f_1}{\partial x_k}(\phi(0),i),\qquad \partial_{kl} V(\phi,i)=\dfrac{\partial^2 f_1}{\partial x_k\partial x_l}(\phi(0),i).$$
}\end{rmk}

Let $V(\cdot,\cdot)\in\BF$, we define the operator
\begin{equation}\label{e:LV}
\begin{aligned}
\LL V(\phi, i)=&V_t(\phi,i)+V_x(\phi,i)b(\phi(0),i)+\dfrac12\trace\Big(V_{xx}(\phi,i)A(\phi(0),i)\Big)\\
&+\sum_{j=1,j\ne i}^\infty q_{ij}(\phi)\big[V(\phi,j)-V(\phi,i\big]\\
=&V_t(\phi,i)+
\sum_{k=1}^nb_k(\phi(0),i)V_k(\phi,i)
+\dfrac12
\sum_{k,l=1}^na_{kl}(\phi(0),i)V_{kl}(\phi,i)\\
&+\sum_{j=1,j\ne i}^\infty q_{ij}(\phi)\big[V(\phi,j)-V(\phi,i)\big],
\end{aligned}
\end{equation}
for any bounded stopping time $\tau_1\leq\tau_2$,
we have the functional It\^o formula:
\begin{equation}\label{f.Ito}
\E V(X_{\tau_2},\alpha(\tau_2))=\E V(X_{\tau_1},\alpha(\tau_1))+\E\int_{\tau_1}^{\tau_2}\LL V(X_s,\alpha(s))ds
\end{equation}
if the expectations involved exist. Equation
\eqref{f.Ito} is obtained by applying the functional It\^o formula for general semimartingales
given in \cite{CF1,CF}
specialized to our processes.

\section{Recurrence and Ergodicity}\label{sec:3}
First, we need some conditions for irreducibility of the process $\{(X_t,\alpha(t)): t\geq0\}$.
\begin{enumerate}[label=(\rm H{\arabic*})]
    \item \label{h1}
\begin{enumerate}
\item     For any $i\in\N,$ $A(x,i)$
is elliptic uniformly on each compact set, that is, for any $R>0$, there is a $\theta_{R,i}>0$ such that
\begin{equation}\label{ellip}
y^\top A(x,i)
y\geq \theta_{R,i}
|y|^2
\ \ \forall |x|\leq R, \
y\in\R^d.
\end{equation}
\item There is an $i^*$ satisfying that for any $i\in\N$, there exist $i_1,\dots,i_k\in\N$ and $\phi_1,\dots,\phi_{k+1}\in\C$ such that
$q_{ii_1}(\phi_1)>0$, $q_{i_l,i_{l+1}}(\phi_{l+1})>0,l=1,\dots, k-1$, and $q_{i_k,i^*}(\phi(k+1))>0$.
\end{enumerate}
\item \label{h2}
 There exists an $i^*\in\N$ such that
\begin{enumerate}
\item
$A(x, i^*)$
is elliptic uniformly on each compact set;
\item for any $(\phi,i)\in\C\times\N$, there exist positive integers $i=i_1,\dots,i_k=i^*$ satisfying 
$q_{i_l,i_{l+1}}(\phi)>0,l=1,\dots, k-1$.
\end{enumerate}

\end{enumerate}

Let
$(X^{\phi,i},\alpha^{\phi,i}(t))$ be the solution to \eqref{e2.3}
with initial data $(\phi,i)\in\C\times\N$.
To simplify the notation, we denote by $\PP_{\phi,i}$ the probability measure conditioned on the initial data $(\phi,i)$,
that is, for any $t>0$,
$$\PP_{\phi,i}\{(X_t,\alpha(t))\in\cdot\}=\PP\{(X_t^{\phi,i},\alpha^{\phi,i}(t))\in\cdot\},$$
and $\E_{\phi,i}$  the expectation associated with $\PP_{\phi,i}$.
To proceed, we state some auxiliary lemmas.

\begin{lm}\label{lm2.2}
Let $\phi\in\C$ and $q_{ij}(\phi)>0$. For any $\eps>0$, there is a $\delta>0$ such that
$$\inf_{\psi\in\C: \|\psi-\phi\|<\delta}\PP_{\psi,i}\{\|X_\delta-\phi\|<\eps, \alpha(\delta)=j\}>0.$$
\end{lm}

\begin{lm}\label{lm2.6}
For any $i>0$, $R>0$ and $\eps>0$,
there is a compact set $\A\in\C$ such that
\begin{equation}\label{e0-lm2.6}
\inf_{\|\phi\|\leq R}\PP_{\phi,i}\{X_r\in\A, \alpha(r)=i\}>0.
\end{equation}
Moreover, if \eqref{ellip} holds for $i$, then for any $k>0$, there is  a $T=T(k,i,R)>0$ such that
\begin{equation}\label{e1-lm2.6}
\inf_{\|\phi\|\leq R}\PP_{\phi,i}\{\|X_t\|>k\,\text{ for some }\,t\in[0,T]\}>0.
\end{equation}
\end{lm}

\begin{lm}\label{lm2.4}
Assume that either  {\rm \ref{h1}} or {\rm\ref{h2}} is satisfied. There is a nontrivial measure $\bnu(\cdot)$ on $\fB(\C)$
such that
$\bnu(\D)>0$ if $\D$ is a nonempty open subset of $\C$
and that
for any $R>0$, $T>r$, there is a $d_{R,T}>0$ satisfying
\begin{equation}\label{e0-lm2.4}
\PP_{\phi,i^*}\{X_{T}\in\B\text{ and } \alpha(T)=i^*\}\geq d_{R,T}\bnu(\B),\,\B\in\fB(\C)\,\text{ given that }\,\|\phi\|\leq R.
\end{equation}
\end{lm}

The three lemmas above will be proved in the appendix.

\begin{lm}\label{lm2.1}
Assume that either {\rm\ref{h1}} or {\rm\ref{h2}} holds.
For any $i\in\N$, there is a $T_{i}>0$
such that for any $T> T_{i}$ and any open set $\B\subset\C$, we have
$$\PP_{\phi,i}\{X_{T}\in\B,\alpha(T)=i^*\}>0,\, \phi\in\C$$
where $i^*$ is as in  {\rm\ref{h1}} or {\rm\ref{h2}} accordingly.
\end{lm}

\begin{proof}
Suppose that {\rm \ref{h1}} holds with
$i=i_1,\dots,i_k=i^*\in\N$ and $\phi_1,\dots,\phi_{k+1}\in\C$ such that
$q_{i_l,i_{l+1}}(\phi_{l+1})>0,l=1,\dots, k-1$.
Since $q_{i_l,i_{1+1}}(\phi_{l+1})>0$,
it follows from Lemma \ref{lm2.2} that
\begin{equation}\label{e1-lm2.1}
\PP_{\psi,i_l}\{\|X_{\eps_l}-\phi_{l+1}\|<1,\alpha(\eps_l)=i_{l+1}\}>0
\,\text{ if } \|\psi-\phi_{l+1}\|<\eps_{l}
\end{equation}
for some $\eps_l\in(0,1)$.
In view of Lemma \ref{lm2.4},
\begin{equation}\label{e2-lm2.1}
\PP_{\psi,i_l}\{\|X_{1+r}-\phi_{l+1}\|<\eps_{l},\alpha(1+r)=i_{l}\}>0
\,\text{ for any } \psi\in\C,
\end{equation}
and
\begin{equation}\label{e5-lm2.1}
\PP_{\psi,i^*}\{X_{1+r+T'}\in\B,\alpha(1+r+T')=i^*\}>0
\,\text{ for any } \psi\in\C, T'\geq0.
\end{equation}
By \eqref{e1-lm2.1}, \eqref{e2-lm2.1}, and the Markov property of $(X_t,\alpha(t))$,
we have
\begin{equation}\label{e3-lm2.1}
\PP_{\psi,i_l}\{\|X_{1+r+\eps_l}-\phi_{l}\|<1,\alpha(1+r+\eps_l)=i_{l+1}\}>0
\,\text{ for any } \psi\in\C.
\end{equation}
Using \eqref{e5-lm2.1},  \eqref{e3-lm2.1},  and applying the Kolmogorov-Chapman equation again, we obtain
\begin{equation}\label{e4-lm2.1}
\PP_{\psi,i}\left\{X_{k(1+r)+\sum\eps_l+T'}\in\B,\alpha\left(k(1+r)+\sum\eps_l+T'\right)=i^*\right\}>0
\,\text{ for any } \psi\in\C.
\end{equation}
The lemma is proved with $T_{i}=k(2+r)\geq k(1+r)+\sum\eps_l$.

Now, suppose that {\rm \ref{h2}} holds,
it follows from Lemma \ref{lm2.2} and the Kolmogorov-Chapman equation
that
\begin{equation}\label{e6-lm2.1}
\PP_{\phi,i}\{\|X_{\eps}-\phi\|<1,\alpha(\eps)=i^*\}>0
\end{equation}
for sufficiently small $\eps$.
The desired result follows from \eqref{e5-lm2.1} and \eqref{e6-lm2.1}
with $T_i=2+r$.
\end{proof}

\begin{lm}\label{lm2.5}
Assume that either {\rm\ref{h1}} or {\rm\ref{h2}} holds. Let
\begin{equation}\label{e2.22}
\eta_k=\inf\{t>0:\|X_t\|\vee\alpha(t)>k\}.
\end{equation}
 Then for any $(\phi,i)\in \C\times\N$, we have $\PP_{\phi,i}\{\eta_k<\infty\}=1, \; \forall k\in \N.$
\end{lm}

\begin{proof}
Suppose that
$p_0=\PP\{\eta_k<\infty\}<1.$
Since $A(x,i^*)$ is elliptic, in view of \eqref{e1-lm2.6}, there is a $T>0$ such that
\begin{equation}\label{e1-lm2.5}
\PP_{\varphi',i^*}\{\eta_k<T\}>0,\; \forall k>1, \|\varphi' \|\leq1.
\end{equation}
In view of Lemma \ref{lm2.1} and \eqref{e1-lm2.5},
$\forall (\psi,j)\in \C\times\N$ there are $T_{\psi,j}, p_{\psi,j}>0$
such that
\begin{equation}\label{e2-lm2.5}
\PP_{\psi,j}\{\eta_k<T_{\psi,j}\}>2p_{\psi,j}.
\end{equation}
Due to the Feller property of $(X_t,\alpha(t))$,
 there exists a $\delta_{\psi,j}>0$ such that
\begin{equation}\label{e3-lm2.5}
\begin{aligned}
\PP_{\psi',i}\{\eta_k\leq T_{\phi,i^*}\}>p_{\psi,j},\;\forall \;\psi'\in\C,\|\psi-\psi'\|<\delta_{\psi,j}.
\end{aligned}
\end{equation}
Since $\sigma(\cdot,i)$ and $b(\cdot,i)$ are locally compact for each $i\in\N$, similar to \cite[Lemma 4.6]{DY}, we can show that there is
an $h_k>0$ such that for any $t>0$,
\begin{equation}\label{e4-lm2.5}
\begin{aligned}
\PP_{\phi,i}\left\{\frac{|X(s)-X(s')|}{(s-s')^{0.25}}\leq h_k, \; \forall\; 0\vee (\eta_k\wedge t-r)\leq s'<s<\eta_k\wedge t\right\}
>\frac{1+p_0}{2}.
\end{aligned}
\end{equation}
Since the set
$$\A_k=\left\{\psi\in\C:|\psi|\leq k,\frac{|\psi(s)-\psi(s')|}{(s-s')^{0.25}}\leq h_k, \forall -r\leq s'<s\leq 0\right\} $$
is compact in $\C,$ we have from \eqref{e3-lm2.5} that there exist $T_k$ and $\tilde p_k$ such that
\begin{equation}\label{e7-lm2.5}
\PP_{\psi,j}\{\eta_n<T_k\}>\tilde p_k>0,\; \forall \; \psi\in \A_k, j\leq k.
\end{equation}
Since $\lim_{t\to\infty}\PP_{\phi,i}\{\eta_k< t\}=p_0<1$, there is a $T'>0$ such that
\begin{equation}\label{e5-lm2.5}
p_0\geq \PP_{\phi,i}\{\eta_k\leq T'\}\geq p_0-\frac{1-p_0}{2}\tilde p_k.
\end{equation}
In view of \eqref{e4-lm2.5} and \eqref{e5-lm2.5}, we have $\PP_{\phi,i}\{T'<\eta_k,X_{T'}\in \A_k\}>\frac{1-p_0}{2}$.
By the Markov property and \eqref{e7-lm2.5},
\begin{equation}\label{e6-lm2.5}
\begin{aligned}
\PP_{\phi,i}\{T'<\eta_k<\infty\}&\geq \PP_{\phi,i}\{X_{T'}\in \A_k, T'<\eta_k\}\\
&\geq \E_{\phi,i}\left[\1_{\{T'<\eta_k,X_{T'}\in \A_k\}}\PP_{X_{T'},\alpha(T')}\big\{\eta_k<\infty\big\}\right]\\
&>\frac{1-p_0}{2}\tilde p_k.
\end{aligned}
\end{equation}
We have from \eqref{e5-lm2.5} and \eqref{e6-lm2.5} that
$$
\begin{aligned}
p_0= \PP_{\phi,i}\{\eta_k<\infty\}=& \PP_{\phi,i}\{\eta_k\leq T'\}+\PP_{\phi,i}\{T'<\eta_k<\infty\}\\
> & p_0-\frac{1-p_0}{2}\tilde p_k+ \frac{1-p_0}{2}\tilde p_k= p_0,
\end{aligned}
$$
 which is a contradiction.
Thus $p_0=1$.
\end{proof}

\begin{deff}{\rm
The process $\{(X_t,\alpha(t)): t\geq0\}$ is said to be recurrent (resp., positive recurrent) relative to a measurable set $\CE\in \C\times\N$
if $$\PP_{\phi,i}\{(X_t,\alpha(t))\in\CE\,\text{ for some } t\geq 0\}=1$$
$$\big(\text{resp. } \E_{\phi,i}\left[\inf\{t>0: (X_t,\alpha(t))\in\CE\}\right]<\infty\big)$$
 for any $(\phi,i)\in\C\times\N$.}
\end{deff}

\begin{thm}\label{thm2.1}
Suppose that either hypothesis {\rm (H1)} or {\rm(H2)} holds. Let $\D$ be a bounded open subset of $\C$ and $N$ be a finite subset of $\N$. If $(X_t,\alpha(t))$ is recurrent relative to $\D\times N$ then $(X_t,\alpha(t))$ is recurrent relative to $\D'\times N'$ for any open set $\D'\subset\C$ and a  finite set $N'\subset \N$ containing $i^*$ with $i^*$ given in either {\rm (H1)} or {\rm (H2)} according to which hypothesis is satisfied.
\end{thm}

\begin{proof}
Let $(\phi_0,i_0)\in\C\times\N$.
In view of Lemma \ref{lm2.6}, there exists a compact set $\A_\D\subset\C$ such that
\begin{equation}\label{e1-thm2.1}
\inf_{\{(\psi,j)\in \D\times N\}}\PP_{\psi,j}\{X_r\in \A_\D, \alpha(r)\in N\}:=p_{\D,N}>0.
\end{equation}
Since $i^*\in N'$, by Lemma \ref{lm2.1}, the Feller property of $(X_t,\alpha(t))$ and the compactness of $\A_\D$,
there is a $T>0$ such that
\begin{equation}\label{e3-thm2.1}
\inf_{\{(\psi,j)\in \A_\D\times N\}}\PP_{\psi,j}\{X_{T}\in \D'\times N'\}\geq\eps_0.
\end{equation}
Define the stopping times
$\vartheta_0=0,\vartheta_{k+1}=\inf\{t>\vartheta_k+T:X_{\eta_{k+1}}\in \D\times N\}.$
By the hypothesis of the theorem, $$\PP_{\phi_0,i_0}\{\vartheta_k<\infty\}=1,\;\forall \; k\in \N.$$
On the other hand,
it follows from \eqref{e1-thm2.1} and \eqref{e3-thm2.1} that
\begin{equation}\label{e4-thm2.1}
\PP_{\psi,j}\{(X_{T},\alpha(T)\in\D'\times N'\}\geq p_{\D,N}\eps_0, \text{ for all } (\psi,j)\in \D\times N.
\end{equation}
Consider the events $$A^k=\{(X_{\vartheta_k+T}\notin\D'\times N'\},\,k\in\N.$$ By the strong Markov property of $(X_t, \alpha(t))$, we have
$$
\begin{aligned}
\PP_{\phi_0,i_0}\left(\bigcap_{k'=k}^\infty A^{k'}\right)=\lim_{l\to\infty}\PP_{\phi_0,i_0}\left(\bigcap_{k'=k}^l A^{k'}\right)\leq\lim_{l\to\infty}{(1-p_{\D,N}\eps_0)}^{l-k}=0.\end{aligned}
$$
Thus
$$\PP_{\phi_0,i_0}\left(\bigcap_{k'=k}^\infty A^{k'}\right)=0.$$
It indicates that the event $\big\{\big(X_{\vartheta_k+T},\alpha(\vartheta_k+T)\big)\in \D'\times N\big\}$ must occur with probability 1.
\end{proof}

\begin{thm}\label{thm2.2}
Suppose that either hypothesis {\rm\ref{h1}} or {\rm\ref{h2}} holds. Let $V(\cdot,\cdot)\in \BF$ such that
$$
\lim_{n\to\infty}\inf\{V(\phi,i):|\phi(0)|\vee i\geq n\}=\infty.$$
Suppose further that there are positive constants $C$ and $H$ such that
\begin{equation}\label{e1-thm2.2}
\LL V(\phi,i)\leq C \1_{\{V(\phi,i)\leq H\}}.
\end{equation}
Then the process $(X_t,\alpha(t))$ is recurrent relative to $\D\times N$, where $\D$ is any open bounded subset of $\C$ and $N\subset\N$ contains $i^*$.
\end{thm}

\begin{proof}
Let $\upsilon_H=\inf\{t>0:V(X_t,\alpha(t))\leq H\}$ and  $\eta_k$ be defined as in Lemma \ref{lm2.5}. In view of Lemma \ref{lm2.5}, $\PP_{\psi,j}\{\eta_k<\infty\}=1, \forall k\in\N$. Let $t>0$. By It\^o's formula
$$\E_{\psi,j}V\Big(X_{t\wedge\upsilon_H\wedge\eta_k},\alpha(t\wedge\upsilon_H \wedge \eta_k)\Big)\leq V(\psi,j).$$
Letting $t\to\infty$, we obtain
$$
\begin{aligned}
&\E_{\psi,j}V\big(X_{\upsilon_H\wedge \eta_k},\alpha(\upsilon_H\wedge\eta_k)\big)\leq V(\psi,j)
,\end{aligned}
$$
which implies
$$
\PP_{\psi,j}\{\upsilon_H>\eta_k\}\leq \frac{V(\psi,j)}{\inf\{V(\phi,i):|\phi(0)|\vee i\geq k\}}.
$$
Letting $k\to\infty$ yields $\PP_{\psi,j}\{\upsilon_H>\eta_k\}\to 0$. Thus,
\begin{equation}\label{e2-thm2.2}
\PP_{\psi,j}\{\upsilon_H<\infty \}=1, \; \forall (\psi,j)\in\C\times \N.
\end{equation}
Now, let $k_0>0$ such that
$\inf\{V(\phi,i):|\phi(0)|\vee i\geq k_0\}\geq 2(H+C+r).$ For any $(\psi,j)\in \C\times\N$ satisfying $V(\psi,j)\leq H$. We have
from \eqref{e1-thm2.2} and It\^o's formula that
$$
\begin{aligned}
\E_{\psi,j}V\big(X_{r\wedge \eta_{n_0}},\alpha(r\wedge\eta_{n_0})\big)\leq H+C+r,
\end{aligned}
$$
which implies
\begin{equation}\label{e4-thm2.2}
\PP_{\psi,j}\{\eta_{n_0}<r\}\leq \frac{H+C+r}{2(H+C+r)}\leq \frac{1}{2}.
\end{equation}
Thus,
\begin{equation}\label{e3-thm2.2}
\PP_{\psi,j}\{\|X_r\|<n_0, \alpha(r)<n_0\}\geq\PP_{\psi,j}\{\eta_{n_0}>r\}>\frac{1}{2}\,\text{ provided } V(\psi,j)\leq H.
\end{equation}

Now, fix $(\phi_0,i_0)\in \C\times \N.$ By \eqref{e2-thm2.2} and Lemma \ref{lm2.5}, we can define almost surely finite stopping times
\begin{equation}\label{e5-thm2.2}
\begin{aligned}
\zeta_1&=\inf\{t\geq 0:\;V(X_t,\alpha(t))\leq H\},\\
\zeta_{2k}&=\inf\{t\geq \zeta_{2k-1}+r:\; |X_t|\vee \alpha(t)\geq n_0\},\\
\zeta_{2k+1}&=\inf\{t\geq \zeta_{2k}:\;V(X_t,\alpha(t))\leq H \}.
\end{aligned}
\end{equation}
Define events $B_k=\{\|X_{\zeta_{2k+1}+r}\|\vee\alpha(\zeta_{2k+1}+r)\leq n_0\}.$
In view of \eqref{e3-thm2.2} and the strong Markov property of $(X_t,\alpha(t))$,
 we can use standard arguments in Theorem \ref{thm2.1} to show that $$\PP_{\phi_0,i_0}\{B_k\text{ occurs for some }k\}=1.$$
 Thus, $(X_t,\alpha(t))$ is recurrent relative to $\{(\phi,i):\|\phi\|\vee i\leq n_0\}$. Combining this with  Theorem \ref{thm2.1} yields the desired result.
\end{proof}

\begin{exam}\label{ex1}
{\rm Let
$$
\begin{aligned}
&q_{12}(\phi)=1, q_{1j}(\phi)=0\,\text{ for }j\geq 3;\\
&q_{i,i-1}(\phi)=C_i+(1+\|\phi\|)^{-1},q_{i,i+1}(\phi)=C_i+(1+\|\phi\|)^{-1}\,\text{ for }i\geq 2,\,C_i\geq0;\\
&q_{ij}(\phi)=0\,\text{ for } i\geq 2, j\notin\{i-1,i,i+1\}.
\end{aligned}
$$
Suppose the switching diffusion is given by
$$dX(t)=\sigma(X(t),\alpha(t))dW(t)-X(t)b(X(t),\alpha(t))dt$$
where $b(x,i)>0, \sigma(x,i)$ are locally Lipchitz in $x$ and uniformly bounded
in $K\times\N$ for each compact set $K\in\R$.
Let $f(x)$ be  twice continuously differentiable such that $f(x)>0$ and $f(x)=|x|$ if $|x|\geq 1.$ Let
\begin{equation}\label{e2-ex1}
\kappa:=\sup_{|x|\leq 1,i\in \N}\left|-\left[\dfrac{d f}{dx}(x)\right]xb(x,i)+\dfrac12\left[\frac{d^2f}{dx^2}f(x)\right]\sigma^2(x,i)\right|<\infty.
\end{equation}
Let $$V(\phi,i)=
f(\phi(0))+2\kappa i.
$$
Direct computation leads to
\begin{equation}\label{e3-ex1}
\LL V(\phi,i)=
\begin{cases}
-\left[ \disp\frac{d f}{dx}(\phi(0))\right]\phi(0)b(\phi(0),i)+\frac{\sigma^2(i)}2\left[\frac{d^2f}{dx^2}(\phi(0))\right]-2\kappa(1+\|\phi\|)^{-1}&\text{ if } i>1,\\
-\left[\disp\frac{d f}{dx}(\phi(0))\right]\phi(0)b(\phi(0),i)+\left[\frac{d^2f}{dx^2}(\phi(0))\right]+2\kappa&\text{ if } i=1.
\end{cases}
\end{equation}
In view of \eqref{e2-ex1} and the fact that $\frac{d }{dx}f(x)=\sgn(x), \frac{d^2}{dx^2}f(x)
=0$ for $|x|\geq1$, $i>1$ we have
\begin{equation}\label{e4-ex1}
\LL V(\phi,i)\leq 0\,\forall \phi\in\C, i>1.
\end{equation}
By \eqref{e3-ex1},
if we assume further
$\lim_{x\to\infty} |x|b(x,1)=\infty$,
then
we can verify that
\begin{equation}\label{e5-ex1}
\LL V(\phi,1)\leq \tilde C_1\1_{\{|\phi(0)|<\tilde H\}}-\tilde C_2,\forall \phi\in\C,
\end{equation}
where $\tilde C_1,\tilde C_2, \tilde H$ are some positive constants.
In view of \eqref{e4-ex1} and \eqref{e5-ex1},
we can easily check that
\eqref{e1-thm2.2} holds
in this example, for $V(\phi,i)$ defined above
and suitable $C, H$.
Thus, if there exists $i^*\in\N$ such that $\sigma(x,i^*)\ne 0$ for any $x\in\R$,
then the conclusion of Theorem \ref{thm2.2} holds for this example.
}\end{exam}

 To proceed, let us recall some technical concepts and results needed to prove the main theorem.
 Let ${\bf\Phi}=(\Phi_0,\Phi_1,\dots)$ be a discrete-time Markov chain on a general state space $(E,\fE)$, where $\fE$ is a countably generated $\sigma$-algebra.
 Denote by $\mathcal{P}$ the Markov transition kernel for ${\bf\Phi}$.
If there is a non-trivial $\sigma$-finite positive measure $\varphi$ on $(E,\fE)$ such that for
any $A\in\fE$ satisfying $\varphi(A)>0$ we have
$$\sum_{n=1}^\infty \mathcal{P}^n(x, A)>0,\, x\in E$$
 where  $ \mathcal{P}^n$ is the $n$-step transition kernel of ${\bf\Phi}$ then the Markov chain ${\bf\Phi}$ is called $\varphi$-\textit{irreducible}.
It can be shown (see \cite{EN}) that if ${\bf\Phi}$ is $\varphi$-irreducible, then there exists a positive integer $d$ and disjoint subsets $E_0,\dots,
E_{d-1}$ such that for all $i=0,\dots, d-1$ and all $x\in E_i$ we have
$$\mathcal{P}(x,E_j)=1 \text{ where } j=i+1 \text{ (mod } d)$$
and
$$ \varphi \left(E\setminus \bigcup_{i=0}^{d-1}E_i\right)=0.$$
The smallest positive integer $d$ satisfying the above is called the period of ${\bf\Phi}.$
An \textit{aperiodic} Markov chain is a chain with period $d=1$.
A set $C\in\fE$ is called \textit{petite} if there exists a non-negative sequence $(a_n)_{n\in\N}$ with $\sum_{n=1}^\infty a_n=1$
and a nontrivial positive measure $\nu$ on $(E,\fE)$
satisfying that
$$\sum_{n=1}^\infty a_n \mathcal{P}^n(x, A)\geq\nu(A),\,\, x\in C, A\in\fE.$$

\begin{lm}\label{lm2.7}
Assume  either {\rm\ref{h1}} or {\rm\ref{h2}} holds. The Markov chain $\{(X_k,\alpha(k)): k\in\N\}$ is irreducible and aperiodic.
Moreover, for every bounded set $\D\in\C$ and a finite set $N\in\N$,
the set $\D\times N$ is petite for $\{(X_k,\alpha(k)): k\in\N\}$.
\end{lm}

\begin{proof}
Similar to \eqref{e4-thm2.1},
there are $k_0\in\N, k_0>r$, $\tilde d_{\D,N}>0$ such that
\begin{equation}\label{e1-lm2.7}
\PP_{\phi,i}\{X_t\in \D, \alpha_{k_0}=i^*\}\geq \tilde d_{D,N}\,\text{ for all }\, (\phi,i)\in \D\times N.
\end{equation}
By the Markov property, we deduce from \eqref{e0-lm2.4} and \eqref{e1-lm2.7} that
for any $k>r$, there exists a $\hat d_{\D,N,k}>0$ such that
\begin{equation}\label{e2-lm2.7}
\PP_{\phi,i}\{X_{k+k_0}\in\B\text{ and } \alpha(k+k_0)i^*\}\geq \hat d_{\D,N,k}\bnu(\B),\,\B\in\fB(\C),\,(\phi,i)\in\D\times N.
\end{equation}
Let $\hat\bnu(\cdot)$ be the measure on $\fB(\C\times\N)$
given by $\hat\bnu(\CE)=\bnu(\{\phi\in\C: (\phi,i^*)\in\CE\})$ for $\CE\in\fB(\C\times\N)$.
Then \eqref{e2-lm2.7} can be rewritten as
\begin{equation}\label{e5-lm2.7}
\PP_{\phi,i}\{(X_{k+k_0},\alpha(k+k_0))\in\CE\}\geq \hat d_{D,N,k}\hat\bnu(\CE),\,\CE\in\fB(\C\times\N),\,(\phi,i)\in\D\times N.
\end{equation}
It can  be checked that
\eqref{e5-lm2.7} implies that
the Markov chain $\{(X_k,\alpha(k)): k\in\N\}$ is $\hat\bnu$-irreducible and
every nonempty bounded set in $\C\times\N$ is petite.
Moreover, suppose that $(X_k,\alpha(k))$ is not aperiodic. Then, there are disjoint set $\CE_0,\dots,\CE_{d-1}, d>1$
such that
\begin{equation}\label{e6-lm2.7}
\hat\bnu\left((\C\times\N)\setminus\bigcup_{j=0}^{d-1}\CE_j\right)=0
\end{equation}
and
$$\PP_{\phi,i}\{(X_1,\alpha(1))\in\CE_j\}=1 \text{ if } j=j'+1 \text{ (mod } d)\,\text{ if } (\phi,i)\in\CE_{j'}
,$$
which results in
\begin{equation}\label{e3-lm2.7}
\PP_{\phi,i}((X_{m},\alpha(m))\in\CE_j)=
\left\{
\begin{array}{ll}
1& \text{ where } m=j+1 \text{ (mod } d)\\
0&\text{ otherwise}\\
\end{array}
\right.
\text{ if } (\phi,i)\in\CE_j.
\end{equation}

In view of \eqref{e2-lm2.7}, for any $m>k_0+r$ and $(\phi,i)\in\C\times\N$,
there is a $\tilde p_{\phi,i,k}>0$ such that
\begin{equation}\label{e4-lm2.7}
\PP_{\phi,i}\{(X_{m},\alpha(m))\in\CE\}\geq \tilde p_{\phi,i,k}\hat\bnu(\CE)
\end{equation}
for any measurable set $\CE\in\fB(\C\times\N)$.
As a result of \eqref{e3-lm2.7} and \eqref{e4-lm2.7},
we have that
$\hat\nu(\CE_j)=0$ for any $j=0,\dots,d-1$.
Thus,
\begin{equation}\label{e7-lm2.7}
\hat\nu\left((\C\times\N)\setminus
\bigcup_{j=0}^{d-1}\CE_j\right)=\hat\nu(\C\times\N)>0,
\end{equation}
which contradicts \eqref{e6-lm2.7}.
This contradiction
implies that $(X_k,\alpha(k))$ is aperiodic.
\end{proof}

\begin{thm}\label{thm2.3}
Suppose that either  {\rm\ref{h1}} or {\rm\ref{h2}} holds.
Let $V(\cdot,\cdot)\in \BF$ such that
\begin{equation}\label{v1}
\lim_{n\to\infty}\inf\{V(\phi,i):|\phi(0)|\vee i\geq n\}=\infty.
\end{equation}
Suppose further that there are positive constants $C_1,C_2$ and $H$ such that
\begin{equation}\label{v2}
\LL V(\phi,i)\leq -C_1+C_2\1_{\{V(\phi,i)\geq H\}}.
\end{equation}
Then, $(X_t,\alpha(t))$ is positive recurrent relative to any  set of the form $\D\times N$ where $\D$ is a nonempty open set of $\C$ and $N\ni i^*$
with $i^*$ given in either {\rm (H1)} and {\rm (H2)}. Moreover, there is a unique invariant probability measure $\mu^*$, and for any $(\phi, i)\in\C\times\N$
$$\lim_{t\to\infty}\|P(t,(\phi,i),\cdot)-\mu^*\|_{TV}=0.$$
\end{thm}

\begin{proof}
Let $\upsilon_H=\inf\{t\geq 0: V(X_t,\alpha(t))\leq H\}$.
In view of the functional It\^o formula,
\begin{equation}\label{e10-thm2.3}
\begin{aligned}
\E_{\phi,i} V(X_{1\wedge\upsilon_H},\alpha(1\wedge\upsilon_H))
=&V(\phi,i)+\E_{\phi,i}\int_0^{1\wedge\upsilon_H}\LL V(X_s,\alpha(s))ds\\
\leq& V(\phi,i)-C_1\E_{\phi,i} 1\wedge\upsilon_H\\
\leq& V(\phi,i)-C_1\PP_{\phi,i}\{\upsilon_H\geq1\}.
\end{aligned}
\end{equation}
For any $t\leq 1$ and $V(\phi,i)\leq H$,
we have
\begin{equation}\label{e11-thm2.3}
\begin{aligned}
\E_{\phi,i} V(X_{t},\alpha(t))
=&V(\phi,i)+\E_{\phi,i}\int_0^{t}\LL V(X_s,\alpha(s))ds\\
\leq& V(\phi,i)+C_2t\\
\leq& H+C_2.
\end{aligned}
\end{equation}
It follows from \eqref{e11-thm2.3}
and the strong Markov property of $(X_t,\alpha(t))$
that
\begin{equation}\label{e12-thm2.3}
\begin{aligned}
\E_{\phi,i}\left[\1_{\{\upsilon_H<1\}}V(X_1,\alpha(1))\right]
\leq& (H+C_2)\PP_{\phi,i}\{\upsilon_H<1\}\\
\leq& 2(H+C_2)-(H+C_2)\PP_{\phi,i}\{\upsilon_H<1\}.
\end{aligned}
\end{equation}
Let $\C_V:=\{(\psi', j'): V(\psi,j)\leq 2(H+C_2)\}$.
In view of \eqref{e10-thm2.3} and \eqref{e12-thm2.3},
\begin{equation}\label{e13-thm2.3}
\begin{aligned}
\E_{\phi,i}V(X_1,\alpha(1))\leq&\E_{\phi,i}\left[\1_{\{\upsilon_H<1\}}V(X_1,\alpha(1))\right]
+\E_{\phi,i} V(X_{1\wedge\upsilon_H},\alpha(1\wedge\upsilon_H))\\
\leq&V(\phi,i)-\min\{C_1,H+C_2\}+ 2(H+C_2)\1_{\{(\phi,i)\in\C_V\}}.
\end{aligned}
\end{equation}
Let $n_0\in\N$ such that
\begin{equation}\label{e9-thm2.3}
V(\phi, i)\geq 2(2H+2C_2+C_2r) \,\text{ for any }\,\|\phi\|\vee i\geq n_0,
\end{equation}
and define $\hat\zeta_V=\inf\{t\geq 0: V(X_t,\alpha(t))\geq 2(2H+2C_2+C_2r)\}$.
Similar to \eqref{e4-thm2.2},
we have
\begin{equation}\label{e2.31}
\begin{aligned}
\PP_{\phi,i}\{\hat\zeta_V\leq r\}\leq& \dfrac12\,\text{ for }\, (\phi,i)\in\C_V.
\end{aligned}
\end{equation}
Thus,
\begin{equation}\label{e2-thm2.3}
\begin{aligned}
\PP_{\phi,i}\{ \|X_r\|\vee\alpha(r)\leq n_0\}&\geq \PP_{\phi,i}\{V(X_r,\alpha(r))\geq H+C_2r+1\}\\
&\geq 1-\PP_{\phi,i}\{\hat\zeta_V\leq r\}=\dfrac{1}2,\,(\phi,i)\in \C_H.
\end{aligned}
\end{equation}
In view of \eqref{e5-lm2.7} and \eqref{e2-thm2.3},
\begin{equation}\label{e3-thm2.3}
\PP_{\phi,i}\{(X_{k+k_0},\alpha(k+k_0))\in\CE\}\geq \hat d_{H,k}\hat\bnu(\CE),\,\CE\in\fB(\C\times\N),\text{ if }V(\phi,i)\leq H, k>r
\end{equation}
for some $\hat d_{H,k}>0$.
Thus, the set $\{(\phi,i)\in\C\times\N: V(\phi,i)\leq H\}$ is petite for $\{(X_k,\alpha(k)):k\in\N\}$.
Using this and  \eqref{e3-thm2.3},
it follows from  \cite[Theorem 2.1]{TT} (or \cite{MT})
that
$$\lim_{n\to\infty}\|P(n,(\phi,i),\cdot)-\mu^*\|_{TV}=0$$
where $P(t, (\phi,i),\cdot)$ is the transition probability of $(X_t,\alpha(t))$
and $\mu^*$ is an invariant probability measure of the Markov chain $\{X_k,\alpha(k),k\in\N\}$.
It is easy to show that $\mu^*$ is also an invariant probability measure of the process $\{(X_t,\alpha(t))\}$.
Thus $\|P(t,(\phi,i),\cdot)-\mu^*\|_{TV}$ is decreasing in $t$,
which leads to
$$\lim_{t\to\infty}\|P(t,(\phi,i),\cdot)-\mu^*\|_{TV}=0.$$

Now we show that the process $(X_t,\alpha(t))$ is positive recurrent.
Similar to \eqref{e10-thm2.3},
we deduce from the functional It\^o formula that
$$\E_{\phi,i}\upsilon_H\leq C^{-1}_1V(\phi,i).$$
Owing to this and \eqref{e3-thm2.3}, we can use the arguments in the proof
of \cite[Lemma 3.6]{ZY}
to show that $(X_t,\alpha(t))$ is positive recurrent.
\end{proof}

\begin{exam}\label{ex2}{\rm
In Example \ref{ex1},
if we assume further that
\begin{equation}\label{e1-ex2}
\lim_{|x|\to\infty}\inf_{i\in\N}\{|x|b(x,i)\}>0
\end{equation}
then it follows from \eqref{e2-ex1} and \eqref{e3-ex1} that
$$\LL V(\phi,i)\leq -\hat C\,\text{ for } \phi\in\C, i\geq 2,$$
for some positive constant $\hat C$.
This combined with \eqref{e5-ex1}
shows that \eqref{v2} holds
for the switching diffusion $(X(t),\alpha(t))$ and the function $V(\phi,i)$ in Example \ref{ex1}.
Thus the conclusion of Theorem \ref{thm2.3} holds
for the switching diffusion in Example \ref{ex1}
with the additional condition \eqref{e1-ex2}.
}\end{exam}

\begin{thm}
Suppose that either  {\rm\ref{h1}} or {\rm\ref{h2}} holds.
Let $V(\cdot,\cdot)\in \BF$ such that
\begin{equation}\label{v3}
\lim_{n\to\infty}\inf\{V(\phi,i):|\phi(0)|\vee i\geq n\}=\infty.
\end{equation}
Suppose further that there are $C_1$ and $C_2>0$ such that
\begin{equation}\label{v4}
\LL V(\phi,i)\leq -C_1V(\phi,i)+C_2.
\end{equation}
Then, there is a unique invariant probability measure $\mu^*$ and $\theta>0$ such that for any $(\phi, i)\in\C\times\N$
\begin{equation}\label{e0-thm2.4}
\lim_{t\to\infty}\exp(\theta t)\|P(t,(\phi,i),\cdot)-\mu^*\|_{TV}=0.
\end{equation}
\end{thm}

\begin{proof}
\begin{equation}\label{e1-thm2.4}
\begin{aligned}
\E_{\phi,i} &\exp\{C_1(\eta_k\wedge t)\}V(X_{\eta_k\wedge t},\alpha(\eta_k\wedge t))\\
=&V(\phi,i)+\E_{\phi,i}\int_0^{\eta_k\wedge t}e^{C_1s}[\LL V(X_s,\alpha(s))+C_1 V(X_s,\alpha(s))]ds\\
\leq& V(\phi,i)+C_2\E_{\phi,i} \int_0^{\eta_k\wedge t}e^{C_1s}ds\\
\leq& V(\phi,i)+C_1^{-1}C_2e^{C_1t}.
\end{aligned}
\end{equation}
Letting $k\to\infty$, we obtain
\begin{equation}\label{e2-thm2.4}
\begin{aligned}
\E_{\phi,i} V(X_t,\alpha(t))
\leq& e^{-C_1t}V(\phi,i)+C_1^{-1}C_2
\end{aligned}
\end{equation}
Let $\gamma_1=e^{-C_1}$ and
$\gamma_2\in(\gamma_1,1)$.
It follows from \eqref{e2-thm2.4} and \eqref{e12-thm2.3} that
\begin{equation}\label{e3-thm2.4}
\begin{aligned}
\E V(X_1,\alpha(1))
\leq& \gamma_1V(\phi,i)+C_1^{-1}C_2\\
=& \gamma_2V(\phi,i)+\Big[C_1^{-1}C_2-(\gamma_2-\gamma_1)V(\phi,i)\Big]\\
\leq& \gamma_2V(\phi,i)+[C_1^{-1}C_2]\1_{\{V(\phi,i)\leq H'\}}
\end{aligned}
\end{equation}
where $H'=C_1^{-1}C_2(\gamma_2-\gamma_1)^{-1}$.
Similar to \eqref{e3-thm2.3},
the set $\{(\phi,i)\in\C\times\N: V(\phi,i)\leq H'\}$ is petite,
which combined with \eqref{e3-thm2.4}
implies the existence of $\gamma_3\in(0,1)$ such that
$$\lim_{n\to\infty}\gamma_3^n\|P(n,(\phi,i),\cdot)-\mu^*\|_{TV}=0$$
due to a well-known theorem (see, e.g., \cite{MT}).
Then \eqref{e0-thm2.4} follows from \eqref{e3-thm2.4}
and the decreasing property of  $\|P(t,(\phi,i),\cdot)-\mu^*\|_{TV}$ in $t$.
\end{proof}

\begin{exam}\label{ex3}{\rm
Suppose that
$$
\begin{aligned}
&q_{12}(\phi)=1, q_{1j}(\phi)=0\,\text{ for }j\geq 3;\\
&q_{i,1}(\phi)=2\int_{-r}^0|\phi(s)|ds,q_{i,i+1}(\phi)=i\int_{-r}^0|\phi(s)|ds\,\text{ for }i\geq 2\\
&q_{ij}(\phi)=0\,\text{ for } i\geq 2, j\notin\{1,i,i+1\}.\\
\end{aligned}
$$
and that the equation for the diffusion part is
$$dX(t)=\sigma(X(t),\alpha(t))dW(t)-b(X(t),\alpha(t))X(t)dt$$
where $\sigma(x,i),b(x,i)$ are locally Lipchitz in $x$ and uniformly bounded
in $K\times\N$ for each compact set $K\in\R$.
Let $V(\phi,i)$ be defined as in Example \eqref{ex1}.
Similar to Examples \ref{ex1} and \ref{ex2}, under the assumption that  $b:=\inf_{(x,i)\in\R\times\N} \{b(x,i)\}>0.$
one can show that \eqref{v4} holds
in this example with this function $V$.
Thus, the conclusion of Theorem \ref{thm2.2} holds for this example.
if there exists $i^*\in\N$ such that $\sigma(x,i^*)\ne 0$ for any $x\in\R$.
}\end{exam}

\begin{exam}\label{ex4}{\rm
Let
$$
\begin{aligned}
&q_{12}(\phi)=1, q_{1j}(\phi)=0\,\text{ for }j\geq 3;\\
&q_{i,i-1}(\phi)=C_i+2|\phi(0)|, q_{i,i+1}=C_i+|\phi(-r)|\,\text{ for }i\geq 2,\,C_i\geq0;\\
&q_{ij}(\phi)=0\,\text{ for } i\geq 2, j\notin\{i-1,i,i+1\}.\\
\end{aligned}
$$
Consider the general equation for diffusion \eqref{eq:sde}, where $\sigma(x,i),b(x,i)$ are locally Lipchitz in $x$
in $K\times\N$ for each compact set $K\in\R$.
Suppose there is a function $U(x):\R^n\mapsto\R_+$
satisfying
\begin{itemize}
\item $U(x)$ is twice continuously differentiable in $x$.
\item  $\lim_{|x|\to\infty}U(x,i)=\infty$.
\item There are positive constants $C_1, C_2,H$ such that
\begin{equation}\label{e2-ex4}
\LL_i U(x)\leq -C_1U(x)+C_2.
\end{equation}
\end{itemize}
Let $V(x,i)=U(x)+i+\int_0^t\exp\{\frac{\ln 2}r(s+r)\}ds.$
By Remark \ref{rmk2.1},
\begin{equation}\label{e3-ex4}
\LL V(\phi,i)=
\begin{cases}
\LL_i U(x)-i-\disp\frac{\ln 2}r\int_0^t\exp\left\{\frac{\ln 2}r(s+r)\right\}ds+2&\text{ if } i>1\\
\LL_i U(x)-\disp\frac{\ln 2}r\int_0^t\exp\left\{\frac{\ln 2}r(s+r)\right\}ds+1&\text{ if } i=1
\end{cases}
\end{equation}
As a consequence of \eqref{e2-ex4} and \eqref{e3-ex4},
there are $C_3$ and $C_4>0$ such that
$$\LL V(\phi,i)\leq -C_3 V(\phi,i)+C_4\,\text{ for }\, (\phi,i)\in\C\times\N.$$
Thus, the conclusion of Theorem \ref{thm2.2} holds for this example
if there exists $i^*\in\N$ such that $A(x,i^*)$ is elliptic.
}\end{exam}

\section{Recurrence of Past-Independent Switching Diffusions}\label{sec:4}

This section is devoted mainly to characterizing the  recurrence of $(X(t),\alpha(t))$
using the corresponding system of partial differential equations
when the switching intensities of $\alpha(t)$ depends only on the current state of $X(t)$,
that is $q_{ij}(\cdot), i, j\in\N$ are functions on $\R^n$ rather than on $\C$.
To simplify the presentation, throughout this section, we set $q_{ii}(x)=0$ for $(x,i)\in\R^n\times\N$.
Thus, $q_{i}(x)=-\sum_{j\in\N}q_{ij}(x)$.
In this section, we use the following assumption.

\begin{asp}\label{asp3.1}{\rm
Suppose that
\begin{enumerate}
  \item either Assumption \ref{asp2.3} or Assumption \ref{asp2.4} holds with $\phi\in\C$ replaced by $x\in\R^n$;
  \item for each $i\in\N$, $A(x,i)$ 
  is uniformly elliptic in each compact set;
  \item for any $x\in\R^n$, there are $\hat q=\hat q(x)>0$ and $n_{\hat q}=n_{\hat q}(x)>0$ such that
\begin{equation}\label{e1-asp3.1}
\sum_{j\leq n_{\hat p}} q_{ij}(x)\geq {\hat q}\ \text{ for any }\ i> n_{\hat q}.
\end{equation}
\end{enumerate}
}\end{asp}

\begin{rmk}{\rm We note the following facts.
\begin{itemize}
\item Part 3 of Assumption \ref{asp3.1} stems from a familiar condition
for uniform ergodicity of the Markov chain having a countable state space.
In other word,
if \eqref{e1-asp3.1} holds, for each $x\in\R^n$,
the Markov chain $\hat\alpha^x(t)$ with generator $Q(x)$ has a property that
$$\sup_{i\in\N} \E_i \varsigma<\infty$$
where $\varsigma$ is the first time the process  $\hat\alpha^x(t)$ jumps to
$\{1,\dots,n_0\}$ and $\E_i$ is the expectation with condition $\hat\alpha^x(0)=i$.
\item Since $q_{ij}(x)$ is continuous in $x\in\R^n$,
with the use of the Heine-Borel covering
theorem, it is easy to show that for any bounded set $D\in\R^n$,
there is $\eps_0=\eps_0(D)$ and $n_0=n_0(D)$ such that
\begin{equation}\label{e2-asp3.1}
\sum_{j\leq n_0} q_{ij}(x)\geq \eps_0\ \text{ for any }\ i> n_0, x\in D.
\end{equation}
\end{itemize} }
\end{rmk}

For an open set $D\subset\R^n$, define
$$\widetilde\tau_D=\inf\{t\geq 0: X(t)\notin D\},$$
and $W^{2,p}_{loc}(D)$
is the set of functions $u:\bar D\mapsto\R$
that has generalized derivatives
$D^\beta u$ for any multiple-index $\beta=(\beta_1,\dots,\beta_n)$ with $|\beta|=\sum \beta_i\leq2 $
satisfying $D^\beta u\in L^p_{loc}(D)$ if $|\beta|\leq 2$.
Let $\HB^p(D)$ be the set of functions $u(x,i)$ in $\bar\D\times\N$
satisfying that
\begin{itemize}
  \item For each $i\in\N$, $u(\cdot,i)\in W^{2,p}_{loc}(D)$ and $u(\cdot,i)$ is continuous in the closure $\bar D$ of $D$.
  \item For any compact set $K\subset \R^n$, $\sup_{(x,i)\in (K\cap\bar D)\times\N}\{u(x,i)\}<\infty$.
\end{itemize}
Let $\LL_i$ be defined as in \eqref{e.L_i}.
We state the two main results of this section.
\begin{thm}\label{thm3.2}
Suppose that Assumption {\rm\ref{asp3.1}} holds.
Let $D_1$ be a bounded open set of $\R^n$ with $\partial D_1\in C^2$ and $D={\bar D_1}^c$ be the complement of $\bar D_1$.
Let $p>n$.
The process $X(t)$ is recurrent relative to $D_1$,  if and only if
the Dirichlet problem
\begin{equation}\label{e1-thm3.2}
\begin{cases}
&\LL_i u(x,i)-q_i(x)u(x,i)+\sum_{j\in\N} q_{ij}(x)u(x,j)=0\text{ in } D\times\N\\
&u(x,i)=f(x,i) \text{ on }\partial D\times\N.
\end{cases}
\end{equation}
has a unique solution in $\HB^p(D)$
given that $f(x,i)$ is continuous in $x\in\partial D$ and
bounded in $\partial\D\times\N$.
\end{thm}

\begin{thm}\label{thm3.1}
Suppose that Assumption {\rm\ref{asp3.1}} holds.
Let $D_1$ be a bounded open set of $\R^n$ with boundary $\partial D_1\in C^2$ and $D=D_1^c$ be its complement.
Let $p>n$.
Suppose further that for each compact set $K\in\R^n$, the function $u(x,i)=\E_{x,i} \widetilde\tau_D$ is bounded  in $K\times\N$.
Then $\{u(x,i)\}\in \HB^p(D), p>n$ is a strong solution to
\begin{equation}
\begin{cases}
&\LL_i u(x,i)-q_i(x)u(x,i)+\sum_{j\in\N} q_{ij}(x)u(x,j)=-1\text{ in } D\times\N\\
&u(x,i)=0 \text{ on }\partial D\times\N.
\end{cases}
\end{equation}
The solution is unique in $\HB^p(D), p>n$.
\end{thm}

\begin{lm}\label{lm3.1}
Let  $D\subset\R^n$ be an  open bounded set with $\partial D\in C^2$.
The Dirichlet problem
\begin{equation}\label{e1-lm3.1}
\begin{cases}
&\LL_i u(x,i)-q_i(x)u(x,i)+\sum_{j\in\N} q_{ij}(x)u(x,j)=f(x,i)\text{ in } D\times\N\\
&u(x,i)\big|_{\partial D}=\phi(x,i) \text{ on }\partial D\times\N.
\end{cases}
\end{equation}
has a unique strong solution $\{u(x,i)\}\in\HB^p(D), p>n$
if $\phi(x,i)$ and $f(x,i)$ are continuous and bounded on $\partial D\times\N$ and $D\times\N$ respectively.
\end{lm}

\begin{proof}
The proof is motivated by that of \cite[Proposition A]{EF}.
However, because there are infinitely many equations,
significant modification is needed.
Let $p>n$ and
\begin{equation}\label{e18-lm3.1}
\hat M=\sup_{(x,i)\in\partial D\times\N} \{|\phi(x,i)|\}+\sup_{(x,i)\in D\times\N} \{|f(x,i)|\}<\infty.
\end{equation}
 By \cite[Theorem 9.1.5]{WYW},  for each $i\in\N$, there exists a strong solution
$u_0(x,i)\in W^{2,p}_{loc}(D)\cap C(\bar D)$ to
\begin{equation}
\begin{cases}
&\LL_i u_0(x,i)-q_i(x)u(x,i)=0\text{ in } D\times\N\\
&u_0(x,i)\big|_{\partial D}=\phi(x,i) \text{ on }\partial D\times\N.
\end{cases}
\end{equation}
Let $Y^{x,i}(t)$ be the solution to
\begin{equation}\label{e2-lm3.1}
dY(t)=b(Y(t), i)dt+\sigma(Y(t),i)dW(t),\ t\geq0
\end{equation}
with initial condition $x$
and $\tau^{x,i}_D=\inf\{t\geq0: Y^{x,i}(t)\notin\D\}$.
In view of the Feyman-Kac formula for diffusion processes
\begin{equation}\label{e13-lm3.1}
\begin{aligned}
u_{0}(x,i)=&\E_{x,i}\left[\phi(Y(\tau_D)),i)\exp\left(-\int_0^{\tau_D}q_i(Y(s))ds\right)\right]\\
&-\E_{x,i}\int_0^{\tau_D}\exp\left(-\int_0^tq_i(Y(s))ds\right)f(Y(t),i)dt.
\end{aligned}
\end{equation}
Note that in \eqref{e13-lm3.1} and what follows,
we drop the superscripts $x$ and $i$ in $Y^{x,i}$ and $\tau^{x,i}_D$
whenever 
the expectation $\E_{x,i}$ or probability $\PP_{x,i}$ is used.
By part (2) of Assumption \ref{asp3.1},
$\sup_{x\in\D}\E_{x,i}\tau_D<\infty$ for any $i\in\N$.
In view of \eqref{e13-lm3.1},
we have
\begin{equation}\label{e16-lm3.1}
|u_0(x,i)|\leq \sup_{x\in \partial\D}\{ |\phi(x,i)|\}+\sup_{x\in \D} \{|f(x,i)|\}\sup_{x\in \D}\{\E_{x,i}\tau_D\}.
\end{equation}
Let $n_0$ and $\eps_0$ satisfy \eqref{e2-asp3.1}. In particular, for $i>n_0$, $q_i(x)\geq \eps_0>0$ in $D$, we can have the following estimate from \eqref{e13-lm3.1}:
\begin{equation}\label{e17-lm3.1}
\begin{aligned}
|u_{0}(x,i)|\leq &\E_{x,i}|\phi_i(Y(\tau_D))|+\E_{x,i}\int_0^{\tau_D}\exp\left(-\int_0^tq_i(Y(s))ds\right)|f(Y(t),i)|dt\\
\leq &\sup_{x\in \partial\D}\{ |\phi(x,i)|\}+\sup_{x\in \D} \{|f(x,i)|\}\E_{x,i}\int_0^{\tau_D}\exp(-\eps_0t)dt\\
\leq &\sup_{x\in \partial\D}\{ |\phi(x,i)|\}+\sup_{x\in \D} \{|f(x,i)|\}\E_{x,i}\int_0^{\tau_D}\exp(-\eps_0t)dt\\
\leq &\sup_{x\in \partial\D}\{ |\phi(x,i)|\}+\eps_0^{-1}\sup_{x\in \D} \{|f(x,i)|\}.
\end{aligned}
\end{equation}
As a result of \eqref{e18-lm3.1}, \eqref{e16-lm3.1}, and \eqref{e17-lm3.1},
\begin{equation}\label{e19-lm3.1}
\begin{aligned}
M_0:=\sup_{(x,i)\in D\times\N} |u_0(x,i)|\leq&\sup_{x\in \partial\D, i\in\N}\{ |\phi(x,i)|\}+\eps_0^{-1}\sup_{x\in \D, i>n_0} \{|f(x,i)|\}
\\&+\sup_{x\in \D, i\leq n_0} \{|f(x,i)|\}\sup_{x\in \D, i\leq n_0}\{\E_{x,i}\tau_D\}<\infty.
\end{aligned}
\end{equation}
Since $u_0(x,i)$ is continuous in $\bar\D\times\N$
and $q_i(x)=\sum_{j\in\N}q_{ij}(x)$
is continuous and bounded in $\bar\D\times\N$,
it is easy to show that $\sum_{j\in\N} q_{ij}u_0(x,j)$ is continuous in $\bar\D\times\N$ and
\begin{equation}\label{e20-lm3.1}
\sup_{(x,i)\in D\times\N} \left|\sum_{j\in\N} q_{ij}u_0(x,j)\right|:= \hat M_0<\infty.
\end{equation}
Thus,
 for each $i\in\N$,
there exists a strong  solution
$u_1(x,i)\in W^{2,p}_{loc}(D)\cap C(\bar D)$  to
\begin{equation}
\begin{cases}
&\LL_i u_{1}(x,i)-q_i(x)u_{1}(x,i)=-\sum_{j\in\N}q_{ij}(x) u_{0}(x,j)\text{ in } D\times\N\\
&u_{1}(x,i)\big|_{\partial D}=\phi(x,i) \text{ on }\partial D\times\N.
\end{cases}
\end{equation}
owing to  \cite[Theorem 9.1.5]{WYW}.
Similar to \eqref{e19-lm3.1}, we can use \eqref{e20-lm3.1} to obtain that
\begin{equation}\label{e21-lm3.1}
\sup_{(x,i)\in D\times\N} |u_1(x,i)|:= M_1<\infty.
\end{equation}
Continuing this way, we can define recursively $\{u_{m+1}(x,i)\}\in\HB^{2,p}(D)$, the strong solution to
\begin{equation}
\begin{cases}
&\LL_i u_{m+1}(x,i)-q_i(x)u_{m+1}(x,i)=-\sum_{j\in\N}q_{ij}(x) u_{m}(x,j)\text{ in } D\times\N\\
&u_{m+1}(x,i)\big|_{\partial D}=\phi(x,i) \text{ on }\partial D\times\N.
\end{cases}
\end{equation}
By the Feyman-Kac formula,
\begin{equation}\label{e3-lm3.1}
\begin{aligned}
u_{m+1}(x,i)=&\E_{x,i}\left[\phi_i(Y(\tau_D)))\exp\left(-\int_0^{\tau_D}q_i(Y(s))ds\right)\right]\\
&+\E_{x,i}\int_0^{\tau_D}\exp\left(-\int_0^tq_i(Y(s))ds\right)\sum_{j\in\N}q_{ij}(Y(t))u_{m}(Y(t),j)dt\\
&-\E_{x,i}\int_0^{\tau_D}\exp\left(-\int_0^tq_i(Y(s))ds\right)f(Y(t),i)dt.
\end{aligned}
\end{equation}
Let $\Delta_m(x,i)=u_{m+1}(x,i)-u_{m}(x,i)$
and $$\Delta^i_m=\sup\{|\Delta_{m+1}(x,i)|: x\in\D\}$$
It follows from \eqref{e3-lm3.1} that
\begin{equation}\label{e4-lm3.1}
\begin{aligned}
|\Delta_{m+1}(x,i)|=&\E_{x,i}\int_0^{\tau_D}\exp\left(-\int_0^tq_i(Y(s))ds\right)\sum_{j\in\N}q_{ij}(Y(t))|\Delta_{m}(Y(t),j)|dt\\
\leq&\sup_{\{i\in\N\}}\{\Delta_m^{i}\}\E_{x,i}\int_0^{\tau_D}\exp\left(-\int_0^tq_i(Y(s))ds\right)q_{i}(Y(t))dt\\
=&\sup_{\{i\in\N\}}\{\Delta_m^{i}\} \E_{x,i}\left[1- \exp\left(-\int_0^{\tau_D}q_i(Y(s))ds\right)\right].
\end{aligned}
\end{equation}
Let
$$p:=\max_{\{i\leq n_0\}}\E_{x,i}\left[1- \exp\left(-\int_0^{\tau_D}q_i(Y(s))ds\right)\right]<1.$$
We have from \eqref{e4-lm3.1} that
\begin{equation}\label{e5-lm3.1}
\begin{aligned}
\max_{\{i\leq n_0\}}\{\Delta_{m+1}^{i}\}\leq p\sup_{\{i\in\N\}}\{\Delta_m^{i}\}.
\end{aligned}
\end{equation}
It also follows from \eqref{e4-lm3.1} that
\begin{equation}\label{e6-lm3.1}
\begin{aligned}
\sup_{\{i\in\N\}}\{\Delta_{m+1}^i\}\leq \sup_{\{i\in\N\}}\{\Delta_m^i\}.
\end{aligned}
\end{equation}
For $i>n_0$, using \eqref{e4-lm3.1} again and then using \eqref{e5-lm3.1} and \eqref{e6-lm3.1}, we have
\begin{equation}\label{e14-lm3.1}
\begin{aligned}
|\Delta_{m+2}(x,i)|\leq& \max_{\{i\leq n_0\}}\{\Delta_{m+1}^i\}\E_{x,i}\int_0^{\tau_D}\exp\left(-\int_0^tq_i(Y(s))ds\right)\sum_{j\leq n_0} q_{ij}(Y(t))dt\\
&+\sup_{\{i> n_0\}}\{\Delta_{m+1}^i\}\E_{x,i}\int_0^{\tau_D}\exp\left(-\int_0^tq_i(Y(s))ds\right)\sum_{j> n_0} q_{ij}(Y(t))dt\\
\leq& p\sup_{\{i\in\N\}}\{\Delta_{m+1}^i\}\E_{x,i}\int_0^{\tau_D}\exp\left(-\int_0^tq_i(Y(s))ds\right)\sum_{j\leq n_0} q_{ij}(Y(t))dt\\
&+\sup_{\{i\in \N\}}\{\Delta_{m+1}^i\}\E_{x,i}\int_0^{\tau_D}\exp\left(-\int_0^tq_i(Y(s))ds\right)\sum_{j> n_0} q_{ij}(Y(t))dt\\
\leq& p\sup_{\{i\in\N\}}\{\Delta_{m+1}^i\}\E_{x,i}\int_0^{\tau_D}\exp\left(-\int_0^tq_i(Y(s))ds\right)\sum_{j\in\N} q_{ij}(Y(t))dt\\
&+(1-p)\sup_{\{i\in \N\}}\{\Delta_{m+1}^i\}\E_{x,i}\int_0^{\tau_D}\exp\left(-\int_0^tq_i(Y(s))ds\right)\sum_{j> n_0} q_{ij}(Y(t))dt.\\
\end{aligned}
\end{equation}
Let
\begin{equation}\label{M_D}
M_D=\sup_{(x,i)\in\D\times N}q_i(x).
\end{equation}
 Note that
$$\dfrac{\sum_{j> n_0} q_{ij}(x)}{q_i(x)}=1-\dfrac{\sum_{j\leq n_0} q_{ij}(x)}{q_i(x)}\leq 1-\dfrac{\eps_0}{M_D}:=\eps_1 \,\text{ for }\, i>n_0,$$
which implies that
\begin{equation}\label{e15-lm3.1}
\begin{aligned}
\E_{x,i}\int_0^{\tau_D}&\exp\left(-\int_0^tq_i(Y(s))ds\right)\sum_{j> n_0} q_{ij}(Y(t))dt\\
\leq& \eps_1\E_{x,i}\int_0^{\tau_D}\exp\left(-\int_0^tq_i(Y(s))ds\right) q_{i}(Y(t))dt\\
\leq&\eps_1\,\qquad\text{ for }\,i>n_0.
\end{aligned}
\end{equation}
In view of \eqref{e14-lm3.1} and \eqref{e15-lm3.1},
we have
\begin{equation}\label{e7-lm3.1}
\sup_{\{i>n_0\}}\Delta^i_{m+2}
\leq [p+(1-p)\eps_1]\sup_{\{i\in \N\}}\{\Delta_{m+1}^i\}.
\end{equation}
By \eqref{e5-lm3.1} and \eqref{e6-lm3.1},
\begin{equation}\label{e8-lm3.1}
\begin{aligned}
\sup_{\{i\leq n_0\}}\{\Delta_{m+2}^i\}\leq p\max_{\{i\in\N\}}\{\Delta_{m+1}^i\}\leq \max_{\{i\in\N\}}\{\Delta_{m}^i\}.
\end{aligned}
\end{equation}
By \eqref{e8-lm3.1} and \eqref{e7-lm3.1},
\begin{equation}\label{e9-lm3.1}
\sup_{\{i\in\N\}}\{\Delta_{m+2}^i\}\leq [p+(1-p)\eps_1]\max_{\{i\in\N\}}\{\Delta_{m}^i\}.
\end{equation}
In view of \eqref{e19-lm3.1} and \eqref{e20-lm3.1},
$\sup_{i\in\N}\{\Delta_1^i\}\leq M_0+M_1<\infty$.
Since $p+(1-p)\eps_1<1$, it follows from \eqref{e6-lm3.1} and \eqref{e9-lm3.1} that
the series $\sum_{m=1}^\infty \sup_{\{i\in\N\}}\{\Delta_{m+2}^i\}$
is convergent.
Thus $u_m(x,i)$ converges uniformly in $(x,i)$
to a function $u(x,i)$.
For each $i\in\N$,
since
$q_i(x)=\sum_j q_{ij}(x)$ is continuous,
the convergence
$\lim_{k\to\infty}\sum_{j<k} q_{ij} (x)=q_i(x)$
is uniform.
Thus,
it is easy to show that as $m\to\infty$,
$\sum_j  q_{ij} (x)u_m(x,j)$ converges uniformly to $\sum_j q_{ij} (x)u(x,j)$, (which is also continuous in $x$).
Using this uniform convergence,
passing the limit in \eqref{e3-lm3.1}
we have
\begin{equation}\label{e10-lm3.1}
\begin{aligned}
u(x,i)=&\E_{x,i}\left[\phi_i(Y(\tau_D)))\exp\left(-\int_0^{\tau_D}q_i(Y(s))ds\right)\right]\\
&+\E_{x,i}\int_0^{\tau_D}\exp\left(-\int_0^tq_i(Y(s))ds\right)\sum_{j\in\N}q_{ij}u(Y(t),j)dt\\
&-\E_{x,i}\int_0^{\tau_D}\exp\left(-\int_0^tq_i(Y(s))ds\right)f(Y(t),i)dt.
\end{aligned}
\end{equation}
Since $\sum_j  q_{ij} (x)u(x,j)$ is continuous in $x$ for each $i\in\N$,
the representation \eqref{e10-lm3.1}
shows that
$u(x,i)$ satisfies
$$ \LL_i u(x,i)-q_i(x)u(x,i)= f(x,i)-\sum_{j\in\N}  q_{ij}(x)u(x,j)\text{ in } D\times\N$$
Since $u_m(x,i)=\phi(x,i)$ on $\partial D\times\N$ for all $m\in\N$,
we have $u(x,i)= \phi(x,i)$ on $\partial D\times\N$.
The existence of solutions is therefore proved.
To prove the uniqueness,
it suffices to consider
the uniqueness in $\HB^p(D)$ of the system
\begin{equation}\label{e11-lm3.1}
\begin{cases}
&\LL_i v(x,i)-q_i(x)v(x,i)+\sum_{j=1}^\infty q_{ij}(x)v(x,j)=0\text{ in } D\times\N\\
&v(x,i)\big|_{\partial D}=0 \text{ on }\partial D\times\N.
\end{cases}
\end{equation}
Let $\{v(x,i)\}\in\HB^p(D)$ be a solution of \eqref{e11-lm3.1}.
Then we have
\begin{equation}\label{e12-lm3.1}
v(x,i)=-\E_{x,i}\int_0^{\tau_D}\exp\left(-\int_0^tq_i(Y(s))ds\right)\sum_{j\in\N}q_{ij}v(Y(s),j)ds.
\end{equation}
Similar to \eqref{e5-lm3.1},
it follows from \eqref{e12-lm3.1} that
$$\sup_{i\leq n_0, x\in D} \{|v(x,i)|\}\leq p\sup_{i\in\N, x\in D} \{|v(x,i)|\}.
$$
Similar to \eqref{e7-lm3.1}, the above inequality and  \eqref{e12-lm3.1}
imply that
$$\sup_{i\in\N, x\in D} \{|v(x,i)|\}\leq [p+(1-p)\eps_1]\sup_{i\in\N, x\in D} \{|v(x,i)|\}.$$
Thus $\sup_{i\in\N, x\in D} \{|v(x,i)|\}=0$,
that is, \eqref{e11-lm3.1} has a unique solution.
\end{proof}

\begin{lm}\label{lm3.2}
Let $D$ be an open and bounded set of $\R^n$.
Let $\xi_0=0$ and
$\xi_k=\inf\{t\geq 0: \alpha(t)\ne\alpha(\xi_{k-1})\},k\in\N$.
Let $f(x,i)$ and $g(x,i)$ are bounded and measurable functions on $D\times\N$
and $\partial D\times\N$ respectively.
Then
\begin{equation}\label{e3-lm3.2}
\begin{aligned}
\E_{x,i}& 1_{\{\xi_1\leq\widetilde\tau_D\}}f(X(\xi_1),\alpha(\xi_1))+\E_{x} 1_{\{\xi_1>\widetilde\tau_D\}}g(X({\widetilde\tau_D}),i)\\
=&
\E_{x,i}\int_0^{\tau_D} q_{ij}(Y(t))f(Y(t),j)\exp\left(-\int_0^{t}q_i(Y(s))ds\right)
\\
&+\E_{x,i} g(Y({\tau_D}),i)\exp\left(-\int_0^{\tau_D}q_i(Y(t))dt\right).
\end{aligned}
\end{equation}
\end{lm}
\begin{proof}
Define
$$
\beta^{x,i}(t)=i+\int_0^t\int_{\R}h(Y^{x,i}_t,\beta^{x,i}(t-), z)\p(dt, dz).
$$
Let $\lambda^{x, i}_1=\inf\{t\geq 0: \beta^{x, i}(t)\ne i\}.$ We have that
\begin{equation}\label{e1-lm3.2}
(X^{x,i}(t),\alpha^{x,i}(t))=(Y^{x,i}(t),\beta^{x,i}(t))\,
\text{ up to }\,\lambda^{x,i}_1=\xi^{x,i}_1,
\end{equation}
where $(X^{x,i}(t),\alpha^{x,i}(t))$ is the solution to \eqref{e2.3} with initial value $(x,i)$
and $\xi^{x,i}_1$ is the first moment of jump for $\alpha^{x,i}(t)$.
Thus,
\begin{equation}\label{e4-lm3.2}
\PP_{x,i}\{\xi_1\wedge\widetilde\tau_D<\infty\}=\PP_{x,i}\{\lambda_1\leq\tau_D\}\geq \PP_{x,i}\{\tau_D<\infty\}=1.
\end{equation}

In view of \cite[Lemma 4.2]{DY},
\begin{equation}\label{e2-lm3.2}
\begin{aligned}
\E_{x,i}& 1_{\{\lambda_1\leq\tau_D\}}f(Y(\lambda_1),\beta(\lambda_1))+\E_{x,i} 1_{\{\lambda_1>\tau_D\}}g(Y({\tau_D}),i)\\
=&
\E_{x,i}\int_0^{\tau_D} q_{ij}(Y(t))f(Y(t),j)\exp\left(-\int_0^{t}q_i(Y(s))ds\right)
\\
&+\E_{x,i} g(Y({\tau_D}),i)\exp\left(-\int_0^{\tau_D}q_i(Y(t))dt\right).
\end{aligned}
\end{equation}
Combining \eqref{e1-lm3.2} and \eqref{e2-lm3.2},
we obtain \eqref{e3-lm3.2}.
\end{proof}

\begin{lm}\label{lm3.3}
Let $D$ be an open bounded set in $\R^n$.
For any $\eps>0$, there is an
$n_2=n_2(\eps)>0$ such that
$$\PP_{x,i}\{\xi_{n_2}\leq\widetilde\tau_D\}<\eps$$
for any $(x,i)\in\B\times\N$.
As a result, $$\PP_{x,i}\{\widetilde\tau_D<\infty\}=1.$$
Moreover,
for any $k>0$, there is a $T>0$ such that
$$\PP_{x,i}\{\xi_{k}\wedge\widetilde\tau_D>T\}<\eps.$$
\end{lm}

\begin{proof}
For each $i\in\N$, we have that
$$p_{i,D}:=\sup_{x\in\D}\E_{x,i}\tau_D<\infty.$$
By Lemma \ref{lm3.2}, with $M_D$ defined as in \eqref{M_D}, we have
\begin{equation}\label{e2-lm3.3}
\begin{aligned}
\PP_{x,i}\{\xi_1>\widetilde\tau_D\}=&\E_{x,i} \exp\left(-\int_0^{\tau_D}q_i(Y(t))dt\right)\\
\geq&\E_{x,i} \exp\left(-M_D\tau_D\right)\\
\geq&\exp\left(-M_D\E_{x,i}\tau_D\right)\\
\geq&\exp\left(-M_Dp_{i,D}\right).
\end{aligned}
\end{equation}
Let $\tilde p:=\min_{\{i\leq n_0\}}\{\exp\left(-M_Dp_{i,D}\right)\}$.
By \eqref{e2-asp3.1}, for $i>n_0$,
$$\dfrac{\sum_{j\leq n_0}q_{ij}(x)}{q_i(x)}\geq \dfrac{\eps_0}{M_D}>0, x\in D.$$
Applying Lemma \ref{lm3.2} again,
we have
\begin{equation}\label{e3-lm3.3}
\begin{aligned}
\PP_{x,i}\{\xi_1\leq\widetilde\tau_D,\,\alpha(\xi_1)\leq n_0\}
=&\E_{x,i}\int_0^{\tau_B} \sum_{j\leq n_0}q_{ij}(Y(t))\exp\left(-\int_0^{t}q_i(Y(s))ds\right)\\
\geq&\dfrac{\eps_0}{M_D}\E_{x,i}\int_0^{\tau_B} q_{i}(Y(t))\exp\left(-\int_0^{t}q_i(Y(s))ds\right)\\
=&\dfrac{\eps_0}{M_D}\PP_{x,i}\{\xi_1\leq\widetilde\tau_D\}.
\end{aligned}
\end{equation}
By the strong Markov property,
\eqref{e2-lm3.3}, and \eqref{e3-lm3.3}, we have
for $i>n_0$ that
\begin{equation}\label{e4-lm3.3}
\begin{aligned}
\PP_{x,i}\{\xi_1\leq\widetilde\tau_D, \xi_2>\widetilde\tau_D\}\geq
&\PP_{x,i}\{\xi_2<\widetilde\tau_D,\,\xi_1\geq\widetilde\tau_D,\,\alpha(\xi_1)\leq n_0\}\\
\geq &\PP_{x,i}\{\xi_1\leq\widetilde\tau_D,\,\alpha(\xi_1)\leq n_0\}\left[\inf_{\{y\in\D,j\leq n_0\}}\PP_{y,j}\{\xi_1<\widetilde\tau_D\}\right]\\
\geq &\dfrac{\tilde p\eps_0}{M_D}\PP_{x,i}\{\xi_1\leq\widetilde\tau_D\}
.\end{aligned}
\end{equation}
Since $\tilde p<1$ and $\dfrac{\eps_0}{M_D}\geq 1$,
we have from \eqref{e3-lm3.3} that
\begin{equation}\label{e5-lm3.3}
\begin{aligned}
\PP_{x,i}\{\xi_2>\widetilde\tau_D\}
\geq& \PP\{\xi_1>\widetilde\tau_D\}+\PP_{x,i}\{\xi_1\leq \widetilde\tau_D, \xi_2>\widetilde\tau_D\}\\
\geq& \PP\{\xi_1>\widetilde\tau_D\}+\dfrac{\tilde p\eps_0}{M_D}\PP_{x,i}\{\xi_1\leq\widetilde\tau_D\}\\
\geq& \dfrac{\tilde p\eps_0}{M_D}\,\text{ for }\, x\in\D, i>n_0.
\end{aligned}
\end{equation}
In light of \eqref{e2-lm3.3},
\begin{equation}\label{e6-lm3.3}
\begin{aligned}
\PP_{x,i}\{\xi_2>\widetilde\tau_D\}
\geq& \PP\{\xi_1>\widetilde\tau_D\}\geq \tilde p\,\text{ for }\, x\in D, i\leq n_0.
\end{aligned}
\end{equation}
Thus, for any $x\in D$ and $i\in\N$, we have
\begin{equation}\label{e7-lm3.3}
\begin{aligned}
\PP_{x,i}\{\xi_2>\widetilde\tau_D\}
\geq& \dfrac{\tilde p\eps_0}{M_D}.
\end{aligned}
\end{equation}
Using the strong Markov property,
we have from \eqref{e7-lm3.3} that
\begin{equation}\label{e8-lm3.3}
\begin{aligned}
\PP_{x,i}\{\xi_{2k}\leq\widetilde\tau_D\}
\leq& \left(1-\dfrac{\tilde p\eps_0}{M_D}\right)^k.
\end{aligned}
\end{equation}
By letting $n_2=2k_2+1$ with $k_2$ being sufficiently large so that
$\left(1-\dfrac{\tilde p\eps_0}{M_D}\right)^{k_2}<\eps$, we complete the proof for the first part of this lemma.

To prove the second part,
note that $\E_{x,i}\tau_D\leq p_{i,D}<\infty$,
thus for any $\eps'>0$, there is $T_1>0$ such that
\begin{equation}\label{e9-lm3.3}
\begin{aligned}
\PP_{x,i}\{\xi_1\wedge\widetilde\tau_D\leq T_1\}
=& \PP_{x,i}\{\lambda_1\wedge\tau_D\leq T_1\}\\
\geq&\PP_{x,i}\{\tau_D\leq T_1\}>1-\eps'\,\text{ for all } x\in\D, i\leq n_0.
\end{aligned}
\end{equation}
For $i>n_0$, we have
\begin{equation}\label{e10-lm3.3}
\begin{aligned}
\PP_{x,i}\{\xi_1\wedge\widetilde\tau_D> T\}
=& \PP_{x,i}\{\tau_D> T, \lambda_1> T\}\\
\geq& \E_{x,i}\left[\1_{\{\tau_D> T\}}\int_0^{T} q_{i}(Y(t))\exp\left(-\int_0^{t}q_i(Y(s))ds\right)dt\right]\\
=& \E_{x,i}\left[\1_{\{\tau_D> T\}}\exp\left(-\int_0^{T}q_i(Y(s))ds\right)\right]\\
\leq &\E_{x,i}\left[\1_{\{\tau_D> T\}}\exp\left(-T\eps_0\right)\right] \,\text{ (since }\, q_i(x)>\eps \text{ if }x\in\D, i>n_0).
\end{aligned}
\end{equation}
Let $T_2>T_1$ such that $\exp(-T_2\eps_0)<\eps'$.
We have from \eqref{e10-lm3.3} that
\begin{equation}\label{e11-lm3.3}
\PP_{x,i}\{\xi_1\wedge\widetilde\tau_D\leq T_2\}>1-\eps'\,\text{ for }\,x\in\D, i>n_0.
\end{equation}
Using \eqref{e9-lm3.3} and \eqref{e11-lm3.3},
\begin{equation}\label{e12-lm3.3}
\PP_{x,i}\{\xi_1\wedge\widetilde\tau_D\leq T_2\}>1-\eps'\,\text{ for }\,x\in\D, i\in\N.
\end{equation}
Using the strong Markov property, it is easy to show that
\begin{equation}\label{e13-lm3.3}
\PP_{x,i}\{\xi_k\wedge\widetilde\tau_D\leq kT_2\}>(1-\eps')^k\,\text{ for }\,x\in\D, i\in\N.
\end{equation}
By choosing $\eps'$ such that $(1-\eps')^k>1-\eps$,
we obtain the second part of this lemma.
\end{proof}

\begin{lm}\label{lm3.4}
Let $D\in\R^n$ be a bounded set.
For $i_0\in\N, T>0,\eps>0$,
there is a $k_0=k_0(i_0,T,\eps)>0$ such that
$$\PP_{x,i_0}\{\zeta_{k_0}>T\}<\eps,\,x\in D,$$
where
$\zeta_k=\inf\{t>0: \alpha(t)\geq k\}.$
\end{lm}

\begin{proof}
This lemma is a direct consequence of \cite[Theorem 4.5]{DY} and the Heine-Borel covering theorem.
\end{proof}
To proceed, we need the following lemma, which is a weak form of Harnack's principle.

\begin{lm}\label{lm3.5}
Let $D$ be an open bounded set in $\R^n$ with $\partial D\in C^2$ and fix $(x_0,i_0)\in\D\times\N$.
Let $B\subset\bar B\subset D$ be a ball centered at $x_0$.
Then for any $\eps>0$, there is a
$c_0=c_0(B, i_0, \eps)>0$ satisfying
$$
u(x,i_0)\leq c_0u(x_0, i_0)+\eps\sup_{\{(y,i)\in\partial D\times\N\}}\{u(y,i)\}, x\in \bar B
,$$
where
$\{u(x,i)\}\in \HB^p(D)$ satisfies
$$ \LL_i u(x,i)-q_i(x)u(x,i)+\sum_{j\in\N}  q_{ij}(x)u(x,j)= 0\text{ in } D\times\N.$$
\end{lm}

\begin{proof}
Let $\phi(x,i)=u(x,i)|_{\partial D}$
and
$$\zeta_k=\inf\{t>0: \alpha(t)\geq k\}.$$
Let $$
u_k(x,i)=
\begin{cases}u(x,i)\,&\text{ if } i<k\\
0&\text{ if } i\geq k.
\end{cases}
$$
By It\^o's formula,
\begin{equation}\label{e5-lm3.5}
\begin{aligned}
\E_{x,i}& u_k(X(\widetilde\tau_D\wedge\zeta_k\wedge t),\alpha(\widetilde\tau_D\wedge\zeta_k\wedge t))\\
=&
u_k(x,i)+\E_{x,i}\int_0^{\widetilde\tau_D\wedge\zeta_k\wedge t}\LL u_k(X(s),\alpha(s))ds\\
=&u_k(x,i)-\E_{x,i}\int_0^{\widetilde\tau_D\wedge\zeta_k\wedge t}\sum_{j\geq k}q_{\alpha(s),j}(X(s))u_k(X(s),j)ds.\\
\end{aligned}
\end{equation}
Letting $k\to\infty$ and then $t\to\infty$,
we obtain from the dominated convergence theorem that
\begin{equation}\label{e1-lm3.5}
u(x,i)=\E_{x,i} \phi(X(\widetilde\tau_D),\alpha(\widetilde\tau_D))
.\end{equation}
As a result of Lemmas \ref{lm3.3} and \ref{lm3.4},
there is a $k_1=k_1(i_0,\eps)\in\N$ such that
\begin{equation}\label{e2-lm3.5}
\PP_{x,i}\{\widetilde\tau_D>\xi_{k_1}\}<\eps.
\end{equation}
In view of \eqref{e1-lm3.5} and \eqref{e2-lm3.5},
\begin{equation}\label{e3-lm3.5}
\begin{aligned}
u(x,i_0)=&\E_{x,i_0} \1_{\{\widetilde\tau_D<\xi_{k_1}\}}\phi(X(\widetilde\tau_D),\alpha(\widetilde\tau_D))+\E_{x,i_0} \1_{\{\widetilde\tau_D>\xi_{k_1}\}}\phi(X(\widetilde\tau_D),\alpha(\widetilde\tau_D))\\
\leq &\E_{x,i_0} \1_{\{\widetilde\tau_D<\xi_{k_1}\}}\phi(X(\widetilde\tau_D),\alpha(\widetilde\tau_D))+\eps\sup_{(y,j)\in\partial D\times\N}\{\phi(y,j)\}.
\end{aligned}
\end{equation}
Let
$$\tilde u(x,i)=\E_{x,i} \1_{\{\widetilde\tau_D<\xi_{k_1}\}}\phi(X(\widetilde\tau_D),\alpha(\widetilde\tau_D))
$$
for $i<k$.
The process $\{(X(t),\alpha(t)), 0\leq t<\xi_{k_1}\}$
can be considered as a switching diffusion process on $\R^n\times\{1,\dots,{k_1}-1\}$
with lifetime $\xi_{k_1}$.
Its generator is
$$\tilde \LL_i f(x,i)=\LL_if(x,i)-q_i(x)u(x,i)+\sum_{j<k_1}q_{ij}(x)f(x,j),$$
for $i=1,\dots,k-1$.
Then \cite[Theorem 3.6]{CZ2} reveals that
$\tilde u(x,i)$ satisfying
\begin{equation}
\begin{cases}
&\tilde \LL_i \tilde u(x,i)=0\text{ in } D\times\{1,\dots,{k_1}-1\}\\
&\tilde u(x,i)\big|_{\partial D}=\phi(x,i) \text{ on }\partial D\times\{1,\dots,{k_1}-1\}.
\end{cases}
\end{equation}
By the Harnack principle for weakly coupled elliptic systems (see e.g., \cite{CZ1}),
there is a $c_0=c_0(k_1)$
such that
\begin{equation}\label{e6-lm3.5}
\tilde u(x,i_0)\leq c_0\tilde u(x_0,i_0)\leq c_0 u(x_0,i_0)
\end{equation}
The desired result follows from \eqref{e3-lm3.5} and \eqref{e6-lm3.5}.
\end{proof}

\begin{rmk}{\rm
In \eqref{e13-lm3.1} and \eqref{e3-lm3.1}, we apply the Feyman-Kac formula
for functions in the class $W^{2,p}_{loc}(D)\cap C(\bar D)$ rather than $C^2(D)$.
Feyman-Kac formula is proved
using It\^o's formula, which is usually stated for $C^2$-functions.
However, It\^o's formula also holds for diffusion processes with functions in $W^{2,p}_{loc}(D)\cap C(\bar D)$ when $p>n$.
The proof for this claim can be found in \cite[Theorem 2.10.2]{NK}.
With a careful consideration, we can generalize the result for diffusion processes to switching diffusion processes in which
the switching  has a finite state space.
Thus, \eqref{e5-lm3.5} holds as long as $u(\cdot,i)\in W^{2,p}_{loc}(D)\cap C(\bar D)$.
}\end{rmk}

\begin{proof}[Proof of Theorem {\rm\ref{thm3.1}}]
Let $k_0\in\N$ sufficiently large such that
$D_1\subset \{x\in\R^n: |x|< k_0\}$.
For $k>k_0$,
define
$D_k=D\cap  \{x\in\R^n: |x|< k\}$.
By \eqref{e1-lm3.5},
$u_k(x,i):=\E_{x,i}\widetilde\tau_{D_k}$
satisfies the equation
\begin{equation}
\begin{cases}
&\LL_i u_k(x,i)-q_i(x)u_k(x,i)+\sum_{j\in\N} q_{ij}(x)u_k(x,j)=-1\text{ in } D_k\times\N\\
&u_k(x,i)\big|_{\partial D}=1 \text{ on }\partial D_k\times\N.
\end{cases}
\end{equation}
Let $B_1\subset B_2$ be two balls in $D$
and
fix $(x_0,i_0)\in B_1\times\N$
and let $k_1>k_0$ be such that
$B_2\subset D_{k_1}$.
Suppose that $\E_{x,i}\widetilde\tau_D<M$ for any $(x,i)\in B_2\times\N$.
Then $u_k(x,i)<M$ for $k>k_0$ and $(x,i)\in B_2\times\N$.
Let $v_{k,m}=u_k(x,i)-u_{m}(x,i)$ for $k>m>k_1$,
we have
\begin{equation}
\LL_i v_{k,m}(x,i)-q_i(x)v_{k,m}(x,i)+\sum_{j\in\N} q_{ij}(x)v_{k,m}(x,j)=0\text{ in } B_2\times\N.
\end{equation}
By Lemma \ref{lm3.5},
for any $\eps>0$, there is a $c_0>0$ such that
\begin{equation}\label{e3-thm3.1}
\begin{aligned}
v_{k,m}(x,i_0)\leq& c_0v_{k,m}(x_0,i_0)+\eps\sup\{v_{k,m}(y,j): (y,j)\in B_2\times\N\}\\
\leq& c_0v_{k,m}(x_0,i_0)+M\eps\,\text{ for any } x\in B_1.
\end{aligned}
\end{equation}
For any $\eps>0$,
since $u_k(x_0,i_0)=\E_{x_0,i_0}\widetilde\tau_{D_k}\to \E_{x_0,i_0}\widetilde\tau_{D}$ as $k\to\infty$,
there exists $k_2=k_2(\eps)$
such that
$c_0v_{k,m}(x_0,i_0)=c_0[u_k(x_0,i_0)-u_m(x_0,i_0)]<\eps$
for any $k>m>k_2$.
In view of \eqref{e3-thm3.1},
\begin{equation}
v_{k,m}(x,i_0)\leq (M+1)\eps\,\text{ for any } (x,i_0)\in B_1\times\N, k>m>k_2
\end{equation}
Thus,
$u_k(x,i_0)$ converges uniformly in each compact subset of $D$.
The limit $u(x,i_0)$ is therefore continuous for any $i_0$.
Now, let  $\phi(x,i)=u(x,i)|_{\partial B_2}$.
Since $\phi(x,i)$ is continuous and uniformly bounded,
by Lemma \ref{lm3.1},
for each $i\in \N$,
there exists $\{\tilde u(x,i)\}\in\HB^p(B_2)$ satisfying
\begin{equation}
\begin{cases}
&\LL_i \tilde u(x,i)-q_i(x)\tilde u_k(x,i)+\sum_{j\in\N} q_{ij}(x)\tilde u(x,j)=-1\text{ in } B_2\times\N\\
&\tilde u(x,i)\big|_{\partial D}=\phi(x,i) \text{ on }\partial B_2\times\N.
\end{cases}
\end{equation}
Similar to \eqref{e1-lm3.5},
by applying It\^o's formula we have that
$$
\begin{aligned}
\tilde u(x,i)=&\E_{x,i}\widetilde\tau_{B_2}+\E_{x,i}\phi(X(\widetilde\tau_{B_2}),\alpha(\widetilde\tau_{B_2}))\\
=&\E_{x,i}\widetilde\tau_{B_2}+\E_{x,i}\E_{X(\widetilde\tau_{B_2}),\alpha(\widetilde\tau_{B_2})}\widetilde\tau_D\\
=&\E_{x,i}\widetilde\tau_D \,\text{ (due to the strong Markov property)}\\
=&u(x,i).
\end{aligned}
$$ The proof is concluded.
\end{proof}

\begin{proof}[Proof of Theorem {\rm\ref{thm3.2}}]
After having Lemma \ref{lm3.5}, we adapt the proof of \cite[Theorem 3.10]{RK}
to obtain the desired result.
First,
suppose that \eqref{e1-thm3.2} has a unique solution in $\HB^p(D)$ for some $p>0$
given that $f(x,i)$ is continuous and bounded on $D\times\N$.
We define $D_k$ as in the proof of Theorem \ref{thm3.1}
and  $v_k(x,i):=\PP_{x,i}\{X(\widetilde\tau_{D_k})\in\partial D\}$.
By \eqref{e1-lm3.5}, $v_k(x,i)$ is the strong solution to
\begin{equation}
\begin{cases}
&\LL_i v_k(x,i)-q_i(x)v_k(x,i)+\sum_{j\in\N} q_{ij}(x)v_k(x,j)=0\text{ in } D_k\times\N\\
&v_k(x,i)\big|_{\partial D}=1 \text{ on }\partial D\times\N.\\
&v_k(x,i)\big|_{\partial D}=0 \text{ on }\{y\in\R^n:|y|=k\}\times\N.
\end{cases}
\end{equation}
By the definition of $v_k(x,i)$, we have that
\begin{equation}
\lim_{k\to\infty}v_k(x,i)=v(x,i):=\PP_{x,i}\{\widetilde\tau_D<\infty\}
\end{equation}
On the other hand, owing to Lemma \ref{lm3.5},
we can use arguments in the proof of Theorem \ref{thm3.1}
to show that $\{v(x,i)\}\in\HB^p(D)$
is the solution to
\begin{equation}\label{e3-thm3.2}
\begin{cases}
&\LL_i v(x,i)-q_i(x)v(x,i)+\sum_{j\in\N} q_{ij}(x)v(x,j)=0\text{ in } D\times\N\\
&v(x,i)\big|_{\partial D}=1 \text{ on }\partial D\times\N.
\end{cases}
\end{equation}
Clearly, $v(x,i)\equiv 1$ is the solution to \eqref{e3-thm3.2}.
By the uniqueness of solutions among the class $\HB^p(D)$,
we have
$$\PP_{x,i}\{\widetilde\tau_D<\infty\}=v(x,i)\equiv 1.$$

Now, suppose that $\PP_{x,i}\{\widetilde\tau_D<\infty\}\equiv 1$
and \eqref{e1-thm3.2} has two solutions $\{v^{(1)}(x,i)\}$
and $\{v^{(2)}(x,i)\}$
for the same $f(x,i)$ being continuous and bounded in $\partial D\times\N$.
Let $v^{(3)}(x,i)=v^{(1)}(x,i)-v^{(2)}(x,i)$.
Then $\{v^{(3)}(x,i)\}\in\HB^p(D)$
and satisfies
\begin{equation}\label{e3-thm3.2}
\begin{cases}
&\LL_i v^{(3)}(x,i)-q_i(x)v^{(3)}(x,i)+\sum_{j\in\N} q_{ij}(x)v^{(3)}(x,j)=0\text{ in } D\times\N\\
&v^{(3)}(x,i)\big|_{\partial D}=0 \text{ on }\partial D\times\N.
\end{cases}
\end{equation}
Let $M^{(3)}=\sup_{(x,i)\in D\times\N}\{v^{(3)}(x,i)\}$.
In view of \eqref{e5-lm3.5}, for $k>k_0\vee|x|$, we have
$$|v^{(3)}(x,i)|=\left|\E_{x,i}\1_{\{|X(\widetilde\tau_{D_k})|=k\}}v^{(3)}(X(\widetilde\tau_{D_k}),\alpha(\widetilde\tau_{D_k}))\right|
\leq M^{(3)}[1-\PP_{x,i}\{X(\widetilde\tau_{D_k})\in\partial D\}]
$$
Letting $k\to\infty$ and using
$\PP_{x,i}\{X(\widetilde\tau_{D_k})\in\partial D\}\to \PP_{x,i}\{\widetilde\tau_D<\infty\}=1$ as $k\to\infty$,
we obtain  $v^{(3)}(x,i)\equiv 0$.
\end{proof}

\appendix
\section{Appendix}\label{sec:apd}
Let $Y^{x,i}(t)$
be the solution to
\begin{equation}\label{e1-A}
dY(t)=b(Y(t), i)dt+\sigma(Y(t),i)dW(t),\ t\geq0
\end{equation}
with initial condition $(x,i)\in\R^n\times\N$.
For $(\phi,i)\in\C\times\N$,
we denote by $Y^{\phi,i}(t), t\geq -r$
be the process satisfying
$Y^{\phi,i}(t)=\phi$ if $t\in[-r,0]$
and $Y^{\phi,i}(t)$ solves \eqref{e1-A} for $t>0$.
Clearly $Y^{\phi,i}(t)=Y^{\phi(0),i}(t)$ for $t\geq0$.
Let $\beta^{\phi,i}$ be the solution to
\begin{equation}\label{e2-A}
\beta^{\phi,i}(t)=i+\int_0^t\int_{\R}h(Y^{\phi,i}_t,\beta^{\phi,i}(t-), z)\p(dt, dz),t\geq0
\end{equation}
satisfying
 $Y^{\phi,i}(t)=\phi(t)$ in $[-r,0]$ and $\beta^{\phi,i}(0)=i$.
Let $\xi^{\phi,i}_1(t)$ and $\lambda^{\phi,i}_1(t)$ be the first jump times
of $\alpha^{\phi,i}(t)$ and $\beta^{\phi,i}(t)$, respectively.
Clearly we have that
\begin{equation}\label{e3-A}
X^{\phi,i}(t)=Y^{\phi,i}(t),\,\alpha^{\phi,i}(t)=\beta^{\phi,i}(t)\,\text{ up to }\, \xi^{\phi,i}_1(t)=\lambda^{\phi,i}_1(t).
\end{equation}

\begin{proof}[Proof of Lemma {\rm\ref{lm2.2}}]
Since $q_{ij}(\cdot)$ is continuous, there is an $\eps\in(0,1)$ such that $q_{ij}(\psi)>0$
given that $\|\psi-\phi\|<\eps$.
Let $M_{\phi}=\sup_{\psi\in\C, \|\psi-\phi\|<1}\{q_i(\psi)\}<\infty$.
Let $\delta_1>0$ such that
\begin{equation}\label{e2-lm2.2}
|\phi(s)-\phi(s')|<\dfrac{\eps}5\,\text{ provided }\,|s-s'|<\delta_1, s,s'\in[-r,0].
\end{equation}
Under either Assumption \ref{asp2.3} or Assumption \ref{asp2.4}, standard arguments show that there exists a sufficiently small $\delta_2\in(0,\delta_1]$ satisfying
\begin{equation}\label{e3-lm2.2}
\PP_{\psi,i}\left\{|Y(t)-\psi(0)|\leq\dfrac{\eps}5\,\forall t\in[0,\delta_2]\right\}\geq\dfrac12,\,\forall \psi\in\C, \|\psi-\phi\|<\eps,
\end{equation}
and
\begin{equation}\label{e4-lm2.2}
\PP_{\psi',j}\left\{|Y(t)-\psi'(0)|\leq\dfrac{\eps}5\,\forall t\in[0,\delta_2]\right\}\geq\dfrac12,\,\forall \psi'\in\C, \|\psi'-\phi\|<\eps.
\end{equation}
In view of \eqref{e2-lm2.2},
it can be checked  that
\begin{equation}\label{e2.7}
\|Y^{\psi,i}_t-\phi\|\leq\dfrac{3\eps}5\,\forall t\in[0,\delta_2] \text{ if }|Y^{\psi,i}(t)-\psi(0)|\leq\dfrac{\eps}5\,\forall t\in[0,\delta_2]\text{
and }\|\psi-\phi\|<\dfrac{\eps}5
\end{equation}
and
\begin{equation}\label{e2.8}
\|Y^{\psi',j}_t-\phi\|\leq\eps\,\forall t\in[0,\delta_2] \text{ if }|Y^{\psi',j}(t)-\psi'(0)|\leq\dfrac{\eps}5\,\forall t\in[0,\delta_2]\text{
and }\|\psi'-\phi\|<\dfrac{3\eps}5.
\end{equation}
By virtue  of \eqref{e3-lm2.2}, \eqref{e2.7}, and  \cite[Lemma 4.2]{DY}, for $\psi\in\C$, $\|\psi-\phi\|<\dfrac{\eps}5$ we have
\begin{equation}\label{e2.9}
\begin{aligned}
& \!\!\! \PP_{\psi,i}\left\{\|Y_{\lambda_1}-\phi\|\leq\dfrac{3\eps}5\text{ and } \lambda_1<\delta_2, \beta(\lambda_1)=j\right\}\\
& \ =\E_{\psi,i} \int_0^{\delta_2} \1_{\{\|Y_t-\phi\|\leq\frac{3\eps}5\}}q_{i,j}(Y_t)\exp\left(-\int_0^tq_i(Y_s)ds\right)dt\\
& \ \geq \E_{\psi,i} \left[\1_{\{|Y(u)-\psi(0)|\leq\frac\eps5\,\forall u\in[0,\delta_2]\}}\int_0^{\delta_2} q_{i,j}(Y_t)\exp\left(-\int_0^tq_i(Y_s)ds\right)dt\right]\\
& \ \geq \dfrac{\delta_2}2 \inf_{\phi'\in\C:\|\phi'-\phi\|\leq\frac{3\eps}5}\big\{q_{i,j}(\phi')\big\}\times\inf_{\phi'\in\C:\|\phi'-\phi\|\leq\frac{3\eps}5}\left\{\exp\left(-\int_0^{\delta_2}q_i(\phi'(s))ds\right)dt\right\}:=p_1>0.
\end{aligned}
\end{equation}
Now, we have from the Markov property that
\begin{equation}\label{e2.10}
\begin{aligned}
\PP_{\psi,i}&\{\|X_{\delta_2}-\phi\|<\eps, \alpha(\delta_2)=j\}\\
&\geq \PP_{\psi,i}\left\{\xi_1<\delta_2, \alpha_1=j, \|X_{\xi_1}-\phi\|<\frac{3\eps}5\right\}\\
&\qquad\qquad\times \PP_{\psi,i}\Big\{\|X_{\delta_2}-\phi\|<\eps, \xi_2>\delta_2\Big|\xi_1<\delta_2,\alpha_1=j, \|X_{\xi_1}-\psi\|<\eps\Big\}.
\end{aligned}
\end{equation}
By \eqref{e3-lm2.2} and \eqref{e2.8},
if $\|\psi'-\phi\|\leq\dfrac{3\eps}5$, then
\begin{equation}\label{e2.11}
\begin{aligned}
& \!\PP_{\psi',j}\{\|X_{t}-\phi\|<\eps\,\forall\, t\in[0,\delta_2], \xi_1>\delta_2\}\\
&\quad = \PP_{\psi',j}\{\|Y_{t}-\phi\|<\eps\,\forall\, t\in[0,\delta_2], \lambda_1>\delta_2\}\\
&\quad\geq \E_{\psi',j}\left[1_{\{\|Y_{t}-\phi\|<\eps\,\forall t\in[0,\delta_2]\}}\exp\left(-\int_0^{\delta_2}q_{j}(Y_s)ds\right)\right]\\
&\quad \geq \PP_{\psi',j}\left\{|Y(t)-\psi'(0)|\leq\dfrac{\eps}5\,\forall t\in[0,\delta_2]\right\}\times\inf_{\phi'\in\C:\|\phi'-\phi\|\leq\eps}\left\{\exp\left(-\int_0^{\delta_2}q_i(\phi'(s))ds\right)\right\}\\
&\quad :=p_2>0.
\end{aligned}
\end{equation}
By the strong Markov property of $(X_t,\alpha(t))$,  applying estimates \eqref{e2.9} and \eqref{e2.11} to \eqref{e2.10}, we obtain
$$\sup_{\psi\in\C:\|\psi-\phi\|\leq\frac\eps5}\PP_{\psi,i}\{\|X_{\delta_2}-\phi\|<\eps, \alpha(\delta_2)=j\}>p_1p_2>0.$$
\end{proof}

\begin{proof}[Proof of Lemma {\rm\ref{lm2.6}}]
Using the Kolmogorov-Centsov theorem \cite[Theorem 2.8]{KS}, for each $i\in\N$ and $R>0$, there exists an $h_{i,R}>0$ such that
\begin{equation}\label{e2-lm2.6}
\PP_{\phi,i}\left\{\sup\limits_{t,s\in[0,r],0<t-s<h_{i,R}}\dfrac{|Y(t)-Y(s)|}{(s-t)^{0.25}}\leq 4\right\}>\dfrac12 ,\ \forall\,\|\phi\|\leq R.
\end{equation}
The detailed justification was given in the proof of \cite[Lemma 4.5]{DY}.
Let $$\A=\left\{\phi\in\C: |\phi(-r)|\leq R, \sup\limits_{t,s\in[-r,0],0<t-s<h_{i,R}}\dfrac{|\phi(t)-\phi(s)|}{(s-t)^{0.25}}\leq 4\right\}.$$
Let $R'>R$ such that $\|\phi\|<R'$ for any $\phi\in\A$
and $M_{R'}=\sup_{\|\phi\|<R'}\{q_i(\phi)\}<\infty$.
For $\|\phi\|\leq R$,
we have that
$$
\begin{aligned}
\PP_{\psi,i}\{X_r\in\A, \xi_1>r\}=&\PP_{\psi,i}\{Y_r\in\A, \lambda_1>r\}\\
=&\E_{\psi,i}\left[\1_{\{Y_r\in\A\}}\int_0^r\exp(-q_i(Y_t))dt\right]\\
\geq &\E_{\psi,i}\1_{\{Y_r\in\A\}}\exp(-r M_{R'})\\
\geq &0.5\exp(-r M_{R'}),
\end{aligned}
$$
which implies \eqref{e0-lm2.6}.

To prove \eqref{e1-lm2.6}, note that $\sup_{(x,i)\in \R^n\times\N}\E_{x,i}\tau_{k}<\infty$
where $\tau_{k}=\inf\{t\geq0: |Y(t)|\geq k\}$, (see e.g., \cite[Corrolary 3.3]{RK} or \cite[Theorem 3.1]{ZY}).
Thus, there is a $T>0$ such that
$$\PP_{x,i}\Big\{\tau_{k}<T\Big\}>\frac{1}{2},\;\forall\,x\in\R^n.$$
Denote $\widetilde\tau_{k}=\inf\{t\geq0: \|X_t\|\geq k\}$.
For $\phi\in\C$ with $\|\phi\|\leq R<k$ we have from \eqref{e3-A} and \cite[Lemma 4.2]{DY} that
$$
\begin{aligned}
\PP_{\phi,i}\Big\{\widetilde\tau_k< T\Big\}\geq&\PP_{\phi,i}\Big\{\widetilde\tau_k<T, \alpha(t)=i\,\text{ for } t\in[0,\widetilde\tau_k)\Big\}\\
=&\PP_{\phi,i}\Big\{\tau_k<T, \beta(t)=i\,\text{ for } t\in[0,\tau_k)\Big\}\\
=&\E_{\phi,i}\left[\1_{\{\tau_k<T\}}\exp\left(-\int_0^{\tau_k}q_i(Y_s)ds\right)\right]\\
\geq&\exp\left(-M_kT\right)\E_{\phi,i}\1_{\{\tau_k<T\}}\geq 0.5\exp\left(-M_kT\right),
\end{aligned}
$$
where $M_k=\sup_{\|\phi\|<k}\{q_i(\phi)\}<\infty$.
The proof is therefore complete.
\end{proof}
We need an auxiliary lemma to obtain Lemma \ref{lm2.4}.
\begin{lm}\label{lm2.3} Fix $i\in\N$ and suppose $A(x, i)$ is elliptic uniformly in each compact subset of $\R^n$.
For $D$ be a bounded open set in $\R^n$ and $K_1, K_2$ be open sets whose closures are contained in $D$.
Then
\begin{equation}\label{e1-lm2.3}
\inf_{\{\phi\in\C: \phi(0)\in K_1\}}\PP_{\phi,i}\big(\{Y(T)\in K_2\}\cap\{Y(t)\in D\forall t\in[0,T]\}\big)>0
\end{equation}
and there is
a  measure $\bnu$ on $\fB(\C)$
such that
$$\PP_{\phi,i}\{Y_{T+r}\in \B, Y(t)\in D, t\in[0,T+r]\}\geq \bnu(\B).$$
Moreover, if $\B\subset\{\phi\in\C:\phi(t)\in D, t\in[-r,0]\}$
is an open set of $\C$, then
$\bnu(\B)>0$.
\end{lm}

\begin{proof}
For a bounded continuous function $f(x):\R^n\mapsto\R$
vanishing outside $K_2$,
let $u_f(t, x)$ be the solution to
\begin{equation}
\begin{cases}
\dfrac{\partial u}{\partial t}+\LL_iu=0 \text{ in } D\times[0,T)\\
u\left(T,x\right)=f(x)\text{ on } D,\\
u(t, x)=0\text{ on } \partial D\times\left[0,T\right].
\end{cases}
\end{equation}
It is well known (see, e.g., \cite[Theorem 2.8.2]{XM}) that
$$u_f(t, x)=\E_{x,i}\left[ f\big(Y(T-t)\big)\1_{\left\{Y(s)\in D\,\forall s\in[0,T-t]\right\}}\right].$$

Let $g$ be a continuous function in $D$ such that $0\leq g(x)\leq 1\,\forall x\in D$, $g(x)=0$ outside $K_2$ and $g(x)>0$ for some $x\in K_2$.
By the strong maximum principle for parabolic equations (see \cite[Theorem 7.12]{LE}),
$u_g(0, x)>0$ for all $x\in D$, which  implies that
\begin{equation}\label{e2.13}
u_C:=\inf\{u_g(0, x): x\in K_1\}>0.
\end{equation}
By the definition of $g(\cdot)$, we can obtain that
\begin{equation}\label{e2.14}
\PP_{x,i}\left\{Y(T)\in K_2, Y(s)\in D\,\forall s\in[0,T]\right\}\geq u_g(0, x)\,\forall x\in D.
\end{equation}
The first desired result follows  from \eqref{e2.13} and \eqref{e2.14}.
Moreover, in view of Harnack's inequality (see  \cite[Theorem 7.10]{LE}),
there is $\tilde\rho_i>0$ such that
$u_f(y, T)\geq \tilde\rho_i u_f(x_0,\frac{T}2)$ for all $y\in K_2$
and $f$ being bounded continuous.
Thus
$$
\E_{x,i}\left[ f\big(Y(T)\big)\1_{\left\{Y(s)\in D\,\forall s\in[0,T-t]\right\}}\right]\geq \rho_i\E_{x_0,i}\left[ f\left(Y\left(0.5T\right)\right)\1_{\left\{Y(s)\in D\,\forall s\in[0,T-t]\right\}}\right]
$$
for any bounded and continuous function $f$.
Thus, we  obtain that
$$
\begin{aligned}
\PP_{x,i}&\left\{ Y(T)\in B \,\text{ and }\, Y(s)\in D\,\forall s\in[0,T]\right\}\\
\geq& \tilde\rho_i\PP_{x_0,i}\left\{ Y\left(0.5T\right)\in B,\,\text{ and }\, Y(s)\in D\,\forall s\in[0,0.5]\right\}\\
\geq&\rho_i\tilde\nu(B)
\end{aligned}
$$
for any Borel set $B$,
where
$$\nu(\cdot)=\PP_{x_0,i}\left\{ Y(0.5T)\in \cdot,\,\text{ and }\, Y(s)\in D\,\forall s\in[0,0.5T]\right\}
$$
and
$\rho_i=\tilde\rho_i \PP_{x_0,i}\left\{ Y(s)\in D\,\forall s\in[0,0.5T]\right\},$
which is positive due to \eqref{e1-lm2.3}.
Denote $\hat\D=\{\phi\in\C: \phi(t)\in D, t\in[-r,0]\}$.
For any Borel set $\B\subset\C$,
we have from the Markov property of $Y^i(t)$
that
$$
\begin{aligned}
\PP_{x_0,i}&\left\{Y_{T+r}\in\B\,\text{ and } Y(s)\in D\,\forall s\in[0,T+r]\right\}\\
=&\PP\left\{Y_{T+r}\in\B\cap\hat\D\Big|Y(T)=y\right\}\PP_{\phi_0,i}\left\{Y(T)\in dy\,\text{ and } Y(s)\in D\,\forall s\in[0,T+r]\right\}\\
\geq&\rho_i\int_{y\in D}\PP\left\{Y_{T+r}\in\B\cap\hat\D\Big|Y(T)=y\right\}\tilde\nu(dy)\\
=&\bnu(\B\cap\hat\D).
\end{aligned}
$$
Now,
let $\B$ be an open subset of $\hat\D$.
Denote $B=\{\phi(-r):\phi\in\B\}$.
Then $B$ is an open subset of $D$.
By the support theorem (see \cite[Theom 3.1]{SV})),
\begin{equation}\label{e4-lm2.3}
\PP_{y,i}\left\{Y_{T+r}\in\B\right\}>0\,\text{
for any }\,y\in B.
\end{equation}
In light of \eqref{e1-lm2.3},
\begin{equation}\label{e5-lm2.3}
\tilde\nu(B)=\PP_{\phi_0,i}\left\{ Y(0.5T)\in B,\,\text{ and }\, Y(s)\in D\,\forall s\in[0,0.5T]\right\}>0.
\end{equation}
In view of \eqref{e4-lm2.3} and \eqref{e5-lm2.3},
$$\bnu(\B)\geq \int_{y\in B}\PP_{y,i}\left\{Y_{T+r}\in\B\right\}\tilde\nu(dy)>0
$$
if $\B$ is an open subset of $\hat\D$.
\end{proof}

\begin{proof}[Proof of Lemma {\rm\ref{lm2.4}}]
Let $D=\{x\in\R^n: |x|<R+1\}$, $K_1=K_2=\{x\in\R^n: |x|<R\}$
and $M_{R,i}=\sup\{q(\phi,i):\|\phi\|\leq R+1\}<\infty$.
In view of Lemma \ref{lm2.3},
$$\PP_{\phi,i^*}\left\{Y_{T+r}\in\B, Y(t)< R+1, t\in[0,T+r]\right\}
\geq\bnu(\B)
$$
where $\bnu(\cdot)$ is defined as in Lemma \ref{lm2.3} with $i$ replaced by $i^*$.
Thus,
$$
\begin{aligned}
\PP_{\phi,i^*}&\{X_{T+r}\in\B\text{ and } \alpha(T+r)=i^*\}\\
\geq&\PP_{\phi,i^*}\{X_{T+r}\in\B\text{ and } \alpha(t)=i^*, \|X_t\|<R+1\,\forall\,t\in[0,T+r]\}\\
\geq&\E_{\phi,i^*}\left[\1_{\{Y_{T+r}\in\B,\|Y_t\|<R\,\forall\,t\in[0,T+r]\}}\exp\left(-\int_0^{T+r}q_{i^*}(Y_t)dt\right)\right]\\
\geq& \exp(-M_{R,i}(T+r))\bnu(\B).
\end{aligned}
$$
The proof is complete.
\end{proof}


\begin{thebibliography}{99}
{\small
\setlength{\baselineskip}{0.05in}

\parskip=0pt
\bibitem{BSY} J. Bao, J. Shao, C. Yuan, Approximation of invariant measures for regime-switching diffusions, {\it Potential Anal.} {\bf44} (2016), no. 4, 707-727.

\bibitem{CZ1} Z.-Q. Chen, Z. Zhao, Harnack principle for weakly coupled elliptic systems, {\it J. Differential Equations}, {\bf139} (1997), no. 2, 261�282.

\bibitem{CZ2} Z.-Q. Chen, Z. Zhao, Potential theory for elliptic systems, {\it Ann. Probab.} {\bf24} (1996), no. 1, 293�319.

\bibitem{CF1} R. Cont,  D.-A. Fourni\'e, Change of variable formulas for non-anticipative functionals on path space, {\it J. Funct. Anal.}, {\bf 259} (2010), no. 4, 1043-1072.

\bibitem{CF} R. Cont, D.-A. Fourni\'e, Functional It\^o calculus
and stochastic integral representation of martingales,
{\it Ann. Probab.}, {\bf41}, (2013) no. 1
109-133.

\bibitem{Dup09}
B. Dupire,  Functional It\^o's Calculus. Bloomberg Portfolio Research Paper No. 2009-04-FRONTIERS Available at SSRN:
http://ssrn.com/abstract=1435551 or http://dx.doi.org/10.2139/ssrn.1435551.

\bibitem{EF} A. Eizenberg, M. Freidlin. On the Dirichlet problem for a class of second order PDE systems with small parameter,
     {\it Stochastics  Stochastic Rep.} {\bf 33}  (1990) no.3-4, 111-148.

\bibitem{LE} L.C. Evans, {\it Partial Differential Equations}, 2nd ed., Graduate Studies in Mathematics, 19. American Mathematical Society, Providence, RI, 2010.

\bibitem{KS} I. Karatzas, S.E. Shreve,  {\it Brownian Motion and Stochastic Calculus},
Springer,
2012.
\bibitem{RK}
R.Z. Khasminskii, {\it Stochastic Stability of Differential Equations},
 2nd ed., Springer, Berlin, 2012.

\bibitem{KZY} R.Z. Khasminskii,  C. Zhu, G. Yin, Stability of regime-switching diffusions, {\it Stochastic Process. Appl.} {\bf117} (2007), no. 8, 1037-1051.

\bibitem{NK} N.V. Krylov, {\it Controlled Diffusion Processes},
Springer,
    2008.


\bibitem{XM} X. Mao,  {\it Stochastic Differential Equations and Applications}, 2nd Ed., Horwood, Chinester, 2008.

\bibitem{MY} X. Mao, C. Yuan. {\it Stochastic Differential Equations with Markovian Switching},  Imperial College Press, London, 2006.

\bibitem{MT}
S.P. Meyn, R.L. Tweedie, Stability of {M}arkovian processes. {I}.
  criteria for discrete-time chains, {\it Adv. in Appl. Probab.}, \textbf{24} (1992),
  no.~3, 542-574.



\bibitem{DY} D.H. Nguyen, G. Yin, {Modeling and analysis of switching diffusion systems:
Past-dependent switching with a countable state space}, {\it SIAM J. Control Optim.} {\bf 54} (2016), no. 5, 2450-2477.


\bibitem{EN}
E. Nummelin, {\it General Irreducible {M}arkov Chains and Nonnegative
  Operators},
  Cambridge Tracts in Mathematics, vol.~83,
  Cambridge University
  Press, Cambridge, 1984.

\bibitem{PH} T. Pang, A. Hussain, An application of functional It\^o's formula to stochastic portfolio optimization with bounded memory, {\it  Proceedings of 2015 SIAM Conference on Control and Its Applications} (CT15),  159-166.


\bibitem{SJ} J. Shao, Criteria for transience and recurrence of regime-switching diffusion processes, {\it Electron. J. Probab.} {\bf20} 2015, no. 63, 15 pp.

\bibitem{SX} J. Shao, F. Xi, Strong ergodicity of the regime-switching diffusion processes. {\it Stochastic Process. Appl.} {\bf 123} (2013), no. 11, 3903-3918.

\bibitem{SX2} J. Shao, F. Xi, Stability and recurrence of regime-switching diffusion processes, {\it SIAM J. Control Optim.} {\bf52} (2014), no. 6, 3496-3516.



\bibitem{SV} Stroock, D. W.,  Varadhan, S. R. On the support of diffusion processes with applications to the strong maximum principle. {\it Proceedings of the Sixth Berkeley Symposium on Mathematical Statistics and Probability},
     1972, 333-359.

\bibitem{TT} P. Tuominen, R. L. Tweedie. Subgeometric rates of convergence of f-ergodic Markov chains, {\it Adv. in Appl. Probab.} {\bf 26}(1994), 775-798.

\bibitem{WYW} Z. Wu, J. Yin,  C. Wang, {\it Elliptic and Parabolic Equations.} Hackensack: World Scientific, 2006.


\bibitem{YZ} G. Yin, C. Zhu, {\it Hybrid Switching Diffusions: Properties and Applications},  Springer, New York,
2010.

\bibitem{ZY} C. Zhu, G. Yin, Asymptotic properties of hybrid diffusion systems, {\it SIAM J. Control Optim.} {\bf 46} (2007), no. 4, 1155-1179.


}
\end{thebibliography}
\end{document}